\documentclass[a4paper, 12pt]{amsart}
\usepackage{dsfont}
\usepackage{amsmath, amsthm, amssymb}
\usepackage{fullpage}
\usepackage{xypic}
\usepackage{graphicx}
\usepackage{color}
\usepackage{mathtools}
\usepackage[usenames,dvipsnames]{xcolor}
\usepackage[latin1]{inputenc}
\usepackage{indentfirst}
\usepackage{epstopdf}
\usepackage{subfigure}
\usepackage{emptypage}
\usepackage{hyperref}
\usepackage{color}

\newcounter{teoremaganso}

\renewenvironment{abstract}{\small\quotation\noindent
 {\bfseries \abstractname .}}{\endquotation \par}

\catcode`\@=11

\def\resetthefootnote{\renewcommand{\thefootnote}{\@arabic\c@footnote} }
\def\@principiremex#1{\trivlist
 \item[\hskip \labelsep{\bfseries #1\ \thetheo}]\ignorespaces}
\def\opar@principiremex#1[#2]{\trivlist
 \item[\hskip \labelsep{\bfseries #1\ \thetheo\ (#2)}]\ignorespaces}

\newcommand{\newTHEOremrom}[2]{\newenvironment{#1}{\refstepcounter{theo}\@ifnextchar[{\opar@principiremex
{#2}}{\@principiremex{#2}}
  }{\qedB\endtrivlist}}
\catcode`\@=12 \DeclareMathSymbol{\square}{\mathord}{AMSa}{"03}
\newcommand{\qedB}{\nopagebreak\hspace*{\fill}$\square$\par}
\newcommand{\Qed}{\nopagebreak\hspace*{\fill}{\vrule width6pt height6pt depth0pt}\par}

\newtheorem {theo} {Theorem} [section]
\newtheorem {theo2} {Theorem}
\newtheorem {problem} [theo2] {Problem}
\newtheorem {prop} [theo] {Proposition}

\newcommand{\N}{\ensuremath{\mathbb{N}}}
\newcommand{\Z}{\ensuremath{\mathbb{Z}}}
\newcommand{\R}{\ensuremath{\mathbb{R}}}

\newcommand{\C}{\ensuremath{\mathbb{C}}}

\definecolor{meu}{rgb}{0.95,0.95,0.95}

\newcommand{\caixa}[1]{\vspace{0.3cm}\noindent
\colorbox{meu}{%
 \parbox{0.99\textwidth} {\color{black}{
{#1}}}}\vspace{0.3cm}}

\begin{document}

\title[Some open problems]{Some open problems\\ in low dimensional dynamical systems}
\author{Armengol Gasull}
\address{Departament de Matem\`{a}tiques, Edifici Cc, Universitat Aut\`{o}noma de Barcelona, 08193 Cerdanyola del Vall\`{e}s (Barcelona), Spain.}
\address{Centre de Recerca Matem\`{a}tica, Edifici Cc, Campus de Bellaterra, 08193 Cerdanyola del Vall\`{e}s (Barcelona), Spain.}
\email{gasull@mat.uab.cat}

\date{}

\maketitle

\begin{abstract}
The aim of this paper is  to share  with the mathematical community
a list of 33 problems  that I have found along the years during my
research. I believe that it is worth to think about them and,
hopefully,  it will be possible either to solve some of the problems
or to make some substantial progress. Many of them are about planar
differential equations but there are also questions about other
mathematical aspects: Abel differential equations, difference
equations, global asymptotic stability, geometrical questions,
problems involving polynomials or some recreational problems with a
dynamical component.

\end{abstract}
% *************************************************************



\section{Introduction}

There are several famous well-known conjectures and open problems,
like for instance   Jacobian conjecture,  Riemman's conjecture,
$3x+1$ conjecture or Collatz problem,  Goldbach's conjecture, or
Hilbert XVI problem, that almost all mathematicians know. Also a
very interesting list of 18 open problems, covering many different
branches of mathematics,  has been  published by Smale, see
\cite{Sma1998}. The aim of this work is much more modest. I will
list several concrete problems that I have found along the years. I
hope that, at least for some of them, it is possible either to solve
or to make some substantial progress.

The problems will be classified in seven categories:  periodic
orbits, period function, piecewise linear systems, Markus-Yamabe and
La Salle problems, geometrical problems, questions involving
polynomials, and recreational questions with a dynamical flavour.
Next we briefly describe them but without precise definitions. In
the corresponding next sections they are contextualized and stated
with more precision.

In Section \ref{se:po} we will propose some questions about the
maximum number of limit cycles of some low dimensional differential
equations, including rigid systems, homogeneous type differential
systems, Li\'{e}nard systems,  Riccati and Abel differential equations,
and a new point of view of Hilbert's XVI problem. Some other related
questions considered in this section are on a second order singular
differential equation, about the maximum number of centers for
polynomial differential systems and on the characterization of some
rational periodic difference equations.

In Section \ref{se:pf} we propose several problems for the period
function of some families of planar systems: a Hamiltonian one, a
system with homogenous components, a third one about the maximum
number of critical periods for planar  polynomial  differential
systems, and we end with the problem proposed by Chicone about the
maximum number of critical periods for quadratic reversible centers
and with a related one about the period function of a family of
reversible equivariant planar differential systems.

Section \ref{se:pls} is devoted to planar piecewise linear systems.
We state some problems about  their number and type  of limit
cycles.

In Section \ref{se:my} we present some questions related with global
asymptotic stability: two problems inspired on the works of Markus,
Yamabe, and La Salle and a third one dealing with linear random
differential or difference equations.

In Section \ref{se:gp} we state three questions with a geometric
component. The first one is well-known and it is about triangular
billiards and the second one is  about the extension of classical
Poncelet's theorem for ellipses to more general algebraic ovals. The
third question is the Loewner's conjecture.

Section \ref{se:pol} includes several problems involving
polynomials. We start with a moments type problem, somehow related
with the Jacobian conjecture, we recall the counterexamples of
Kouchnirenko's conjecture about the number of solutions of
fewnomials systems and propose an alternative question, and we end
stating the Casas-Alvero's conjecture.

Finally, in Section \ref{se:con} we collect three known conjectures
with some dynamical flavour: the conjecture of multiplicative
persistence, the 196 conjecture and Singmaster's conjecture.

\section{Periodic orbits}\label{se:po}

The celebrated Hilbert XVIth problem, about the number of limit
cycles of planar polynomials differential systems, has been
extensively studied during the last century and also the beginning
of the current one, see for instance the surveys
\cite{Ily2002,Li2003}. In this section, we present some related
problems for some particular families of differential systems. We
hope that advancing in these simpler cases can give some light to
tackle the general question. Some related and complementary papers,
collecting also some open problems are
\cite{ChaGra2003,ChaSab1999,Gin2007,LliZha2020}. This section also
contains some  questions about periodic orbits on different
contexts.

\subsection{Low degree rigid systems.} \label{se:rig}

Rigid systems are planar systems such that in polar coordinates
their associated angular differential equation is  $\dot \theta =1$.
The origin is their only equilibrium point,  and their limit cycles,
if exist are all nested. Moreover, centers are also isochronous
centers. They were introduced by Conti (\cite{Con1994}) and
afterwards they have been studied by many authors. They write as
\begin{equation}\label{eq:rig}
\begin{cases}
\dot x=-y+xF(x,y),\\
\dot y=\phantom{-}x+yF(x,y),
\end{cases}
\end{equation}
where $F$ is an arbitrary smooth function. Moreover, when $F$ is a
polynomial of degree~$n,$ $F=F_0+F_1+\cdots+ F_n,$ where $F_j$ are
homogeneous polynomials of degree $j,$ and in polar coordinates they
write as
\begin{equation}\label{eq:pol}
\frac{{\rm d} r}{{\rm d} \theta}=r'=\sum_{j=0}^n F_j(\cos
\theta,\sin\theta)r^{j+1}.
\end{equation}
Notice that this last expression is a $2\pi$-periodic non-autonomous
differential equation of Abel type. Its positive  $2\pi$-periodic
solutions are precisely the periodic solutions of \eqref{eq:rig}.

It is not difficult to see that when $n=1,$ that is  $F=F_0+F_1,$
system \eqref{eq:rig} has not limit cycles. Let us prove this
assertion by contradiction. Let $\gamma$ be a periodic orbit of
system~\eqref{eq:rig}. This periodic orbit is transformed into a
positive $2\pi$ periodic solution of   the Riccati differential
equation \eqref{eq:pol}, $r=r(\theta).$ Dividing \eqref{eq:pol} by
$r^2$ and writing $F_1(x,y)=bx+cy$ we get that
\[
\frac{r'(\theta)}{r^2(\theta)}=\frac{F_0}{r(\theta)}
+b\cos\theta+c\sin\theta.
\]
By integrating between $\theta=0$ and $\theta=2\pi,$
\[
0=\frac1{r(0)}-\frac1{r(2\pi)}=\!\int_0^{2\pi}
\frac{r'(\theta)}{r^2(\theta)}\,{\rm d}\theta= \!\int_0^{2\pi}
\frac{F_0}{r(\theta)}\,{\rm d}\theta+
\int_0^{2\pi}\!\!\!\big(b\cos\theta+c\sin\theta\big)\,{\rm
d}\theta=\int_0^{2\pi}\frac{F_0}{r(\theta)}\,{\rm d}\theta.
\]
Therefore we get a contradiction, unless $F_0=0.$ Hence we have
proved that when $F_0\ne0$ system \eqref{eq:pol} with $F=F_0+F_1$
has not periodic orbits. When $F_0=0$ it can have periodic orbits,
but not limit cycles. This is so because, in this case, Riccati
equation \eqref{eq:pol} is of separable variables and it can be
easily integrated.

Equation \eqref{eq:pol} when $n=2$ is precisely an Abel differential
equation and it is the case we are interested. It writes as
\begin{equation}\label{eq:rig2}
\begin{cases}
\dot x=-y+x(a+bx+cy+dx^2+exy+fy^2),\\
\dot y=\phantom{-}x+y(a+bx+cy+dx^2+exy+fy^2).
\end{cases}
\end{equation} In \cite{GasProTor2005}, examples with two limit cycles are
given. For instance a way for obtaining two limit cycles is by a
degenerate Andronov-Hopf bifurcation, because the first Lyapunov
constants for system \eqref{eq:rig2} are
\[
V_1={\rm e}^{2\pi a}-1,\quad V_3=\pi(d+f),\quad
V_5=\pi\big((c^2-b^2)d-bce\big)/2.
\]
Moreover the system has a center if and only if $V_1=V_3=V_5=0.$

Associated to $F_2$ and following again \cite{GasProTor2005} we
define the discriminant $\Delta:=e^2-4df.$ Then it holds that when
$\Delta\le0$ system \eqref{eq:rig2} has at most one limit cycle and
when it exists it is hyperbolic. This is so because under this
hypothesis the coefficient  of $r^3,$  $F_2(\cos\theta,\sin\theta),$
of the Abel differential equation \eqref{eq:pol} when $n=2,$ does
not change sign and when $d^2+e^2+f^2\ne0$ is not identically zero.
Then, following \cite{GasLli1990,Lin1980} it holds that this Abel
equation has at most three periodic orbits, taking into account
their multiplicities. Finally, since $r=0$ is always one of these
periodic orbits and, by symmetry of the equation, if $r(\theta)$ is
one periodic orbit then $-r(\theta+\pi)$ it is also another one, we
get that equation \eqref{eq:pol}, when $n=2,$ has at most one
positive periodic orbit, which has multiplicity one. This fact
implies that when $\Delta\le0$ system~\eqref{eq:rig2} has at most
one (hyperbolic) limit cycle, as we wanted to prove.

Hence, the left open problem reduces to the case $\Delta>0.$ In this
case and without loss of generality,   parameter $f$ can be taken as
$0$ via a linear change of variables. Moreover, it is not
restrictive to take $d\in\{0,1\}.$

\caixa{
\begin{problem}
Consider the family of planar cubic rigid systems
$$
\begin{cases}
\dot x=-y+x(a+bx+cy+dx^2+exy),\\
\dot y=\phantom{-}x+y(a+bx+cy+dx^2+exy).
\end{cases}
$$
Is 2 its maximum number of limit cycles?
\end{problem}
}

\subsection{Systems with homogeneous components.}\label{se:hom}

We start presenting next result given in \cite{CimGasMan1997}, with
a slightly different proof.

\begin{theo}
Consider system
\begin{equation}\label{eq:hom}
\dot x=P_n(x,y),\quad \dot y=Q_m(x,y),
\end{equation}
where $P_n$ and $Q_m$ are  homogeneous polynomials of degrees $n$
and $m,$ respectively. If it has limit cycles then $n\ne m$ and both
$n$ and $m$ are odd. Moreover, in this case, there are polynomials
$P_n$ and $Q_m$ such that system \eqref{eq:hom} has at least
$(n+m)/2$ limit cycles.
\end{theo}

\begin{proof} Because $P_n$ and $Q_m$ are homogeneous, if
$(x_0,y_0)$ is a equilibrium point of \eqref{eq:hom} different of
the origin then, all the real line $\lambda(x_0,y_0),$
$\lambda\in\R,$ is full of equilibrium points. Therefore, if
\eqref{eq:hom} has some periodic orbit,  the origin must be the only
equilibrium point of the system. Moreover, by using
\cite{EisLev1977} we know that its index $\operatorname{ind}(0,0)$
satisfies
\[
\operatorname{ind}(0,0) \equiv nm \pmod{2}.
\]
It is well-known that if a periodic orbit  surrounds a unique
equilibrium point then its index must be 1. Hence,  $nm \equiv 1
\pmod{2}$ and as a consequence $n$ and $m$ must be odd as we wanted
to prove.

It is also well-known  that if $n=m$ then \eqref{eq:hom} can have
periodic orbits, but not limit cycles. This is so, because by
homogeneity, if $\gamma$ is a periodic orbit of the systems  all
orbits homothetic to $\gamma$ are as well periodic orbits and hence
$\gamma$  is not an isolated periodic orbit, see also Section
\ref{se:shc}.

Hence the first part of the theorem is already proved. To prove the
second part, recall that for general $\mathcal{C}^1$ perturbed
Hamiltonian systems,
\begin{equation}\label{eq:gen}
 \begin{cases}
     \dot x=\phantom{-}\dfrac{\partial H(x,y)}{\partial y}+\varepsilon R(x,y,\varepsilon),
     \\[10pt] \dot y=-\dfrac{\partial H(x,y)}{\partial x}
     +\varepsilon S(x,y,\varepsilon),
 \end{cases}
\end{equation}
where $\varepsilon$ is an small parameter, its associated
Melnikov--Poincar\'{e}--Pontryagin function is
\[
M(h)=\int _{\gamma(h)} S(x,y,0)\,{\rm d}x-R(x,y,0)\,{\rm d}y=\sigma
\iint_{G(h)} \frac{\partial R(x,y,0)}{\partial x} +\frac{\partial
S(x,y,0)}{\partial y} \,{\rm d}x{\rm d}y,
\]
where $\sigma$ is $\pm1$ according the orientation of the time
parameterization of $\gamma(h).$  Here, the curves $\gamma(h)$ form
a continuum of ovals contained in $\{H(x,y)=h,$ for $h\in
(h_0,h_1)\},$ and the second expression is only valid if for $h=h_0$
the oval reduces to a point and $G(h)$ is the region surrounded by
$\gamma(h),$ see for instance \cite{ChiLi2007,DumLliArt2006}. It is
known that each simple zero $ h^*\in (h_0,h_1)$ of $M$ gives rise to
a limit cycle of \eqref{eq:gen} that tends, when $\epsilon\to0,$ to
$\gamma(h^*).$

We write $n=2k-1\ge1$ and $m=2\ell-1\ge1,$ where  without loss of
generality $k> \ell,$ and consider,
\begin{equation}\label{eq:gen2}
 \begin{cases}
     \dot x=\phantom{-}y^{2k-1}+\varepsilon P_{2k-1}(x,y)=y^{2k-1}
     +\varepsilon  \displaystyle\sum_{j=1}^{2k-1} \frac{a_j}j y^{2k-1-j}x^j,
     \\[14pt] \dot y=-x^{2\ell-1}
     +\varepsilon Q_{2\ell-1}(x,y)=-x^{2\ell-1}+\varepsilon
     \displaystyle\sum_{j=1}^{2\ell-1} \frac{b_j}j x^{2\ell-1-j}y^j.
 \end{cases}
\end{equation}
Then  $H(x,y)=\frac{x^{2\ell}}{2\ell}+\frac{y^{2k}}{2k},$
$(h_0,h_1)=(0,\infty)$ and
\[
M(h)=\sum_{j=1}^{2k-1} a_j \iint_{G(h)}
y^{2k-1-j}x^{j-1}\,dxdy+\sum_{j=1}^{2\ell-1} b_j \iint_{G(h)}
x^{2\ell-1-j}y^{j-1}\,dxdy.
\]
By symmetry of the sets $G(h)=\{x^{2\ell}/(2\ell)+y^{2k}/(2k)\le
h\},$ when $j$ is even all the above integrals identically vanish.
So we can write $j=2i+1.$ Hence,
\[
M(h)=\sum_{i=0}^{k-1} a_{2i+1} \iint_{G(h)}
y^{2(k-i-1)}x^{2i}\,dxdy+\sum_{i=0}^{\ell-1}  b_{2i+1} \iint_{G(h)}
x^{2(\ell-i-1)}y^{2i}\,dxdy.
\]
We introduce $w$ such that $h=w^{2k\ell}.$ Then, by using the change
of variables $x=w^k X$ and $y=w^\ell X,$ we get that
\[
\iint_{G(h)} x^{2r}y^{2s}\,dxdy=  w^{2(k r+\ell s)+ k+\ell}
\iint_{G(1)} x^{2r}y^{2s}\,dxdy=:I_{r,s}w^{2(k r+\ell s)+k+\ell}.
\]
Hence,
\begin{align*}
M(h)&=\sum_{i=0}^{k-1}  a_{2i+1}I_{i,k-i} w^{2k\ell+(2i+1)(k-\ell)}
+\sum_{i=0}^{\ell-1}
b_{2i+1}I_{\ell-i,i}w^{2k\ell+(2i+1)(\ell-k)}\\&=
w^{2k\ell}\left(\sum_{i=0}^{k-1}  a_{2i+1}I_{i,k-i} \rho^{2i+1}
+\sum_{i=0}^{\ell-1}
b_{2i+1}I_{\ell-i,i}\rho^{-(2i+1)}\right)\\&=w^{2k\ell+k+\ell}\rho^{1-2\ell}\left(\sum_{i=0}^{k-1}
 a_{2i+1}I_{i,k-i} \rho^{2(\ell+i)} +\sum_{i=0}^{\ell-1}
b_{2i+1}I_{\ell-i,i}\rho^{2(\ell-1-i)}\right)\\&=:w^{2k\ell+k+\ell}\rho^{1-2\ell}\sum_{j=0}^{k+\ell-1}
c_j \big(\rho^2\big)^j,
\end{align*}
where $\rho=w^{2(k-\ell)}$ and $c_j$ are arbitrary constants, given
in terms of the parameters $a_r$ and $b_s.$ Therefore, taking
suitable values of these parameters we get any  polynomial of degree
$k+\ell-1$ in $\rho^2.$ Choosing it with all its roots  positive and
simple we obtain that $M$ has $k+\ell-1$ simple positive roots and
as a consequence a system  of the form \eqref{eq:gen2}, which
clearly belong to family \eqref{eq:gen}, such that for $\varepsilon$
small enough has $k+\ell-1=(n+m)/2$ limit cycles.
\end{proof}

From the above result a natural question is:

\caixa{
\begin{problem} Is $(n+m)/2$ the maximum number of limit cycles of
    $$
      \dot x=P_n(x,y),\quad \dot y=Q_m(x,y)
    $$
    where $n\ne m$ and $P_n$ and $Q_m$ are  homogeneous polynomials of odd
    degrees $n$ and $m,$ respectively?
\end{problem}}

A first challenge in  the above problem is to deal with the simplest
case, that corresponds to $n=1$ and $m=3.$

\caixa{
\begin{problem} (i) Consider the cubic family
    \begin{equation}\label{eq:1-3}
    \begin{cases}
    \dot x=ax+by,\\
    \dot y=cx^3+dx^2y+exy^2+fy^3.
    \end{cases}
    \end{equation}
Is $2$ its maximum number of limit cycles?

    (ii) Give a simple proof, if it is true, of the uniqueness of the limit
    cycle for equation
    \begin{equation}\label{eq:1-3p}
    \begin{cases}
    \dot x=y,\\
    \dot y=-x^3+dx^2y+y^3.\qquad
    \end{cases}
    \end{equation}
\end{problem}}

Because of the difficulty to deal with \eqref{eq:1-3} some efforts
have been spend with the even more concrete system \eqref{eq:1-3p}.
For it is known that:
\begin{itemize}
\item It has not limit cycles  for $d\ge0$ and $d<-2.679,$
see \cite{GasGia2006}.

\item It has at most one limit cycle
$ -2.381<d<0,$ see \cite{GasGia2006}.

\item It has at least one limit cycle for
$-2.110<d<0,$ see \cite{GiaGra2015}.

\item A numerically study seems to reduce the range of existence of limit cycles to $-2.198<d<0.$

\end{itemize}

The  results of the first two items are obtained by using suitable
Dulac functions, see next section for more details about this
approach. The system is also studied with the same tool for some
values of $d$ in \cite{CheGri2010}.

\subsection{Low degree classical Li\'{e}nard systems.}\label{se:lie}

Li\'{e}nard equations $\ddot x+f(x)\dot x+ x=0$ are a subject of
continuous study and for many functions $f$ present isolated
oscillations. Maybe the most famous  one is the  van der Pol
equation, for which $f$ is a cubic polynomial. These oscillations
can be seen as limit cycles of the associated planar system:
\begin{equation}\label{eq:lieg}
\begin{cases}
\dot x=y-F(x),\\
\dot y=-x,
\end{cases}
\end{equation}
with $F'(x)=f(x)$ and $F(0)=0.$

When $F$ is a polynomial of degree $n$ its maximum number of limit
cycles, say  $\operatorname{Lie}(n),$ is not known in general.
During many years people tried to prove the conjecture of Lins,
de~Melo and Pugh (\cite{LinMelPug1977}) that asserted that
$\operatorname{Lie}(n)=[(n-1)/2],$ where $[\,\,]$ denotes the
integer part function. This conjecture has been proved to  be false
for $n=7$ in \cite{DumPanRou2007} and later, counterexamples for any
$n\ge 6,$ with 2 more limit cycles that the conjectured number, have
been given in \cite{MaeDum2011}. These counterexamples were found by
studying slow-fast Li\'{e}nard systems. Nowadays, no upper bound for
arbitrary $n$ is neither known nor conjectured. For more detailed
information, see also the survey paper \cite{LliZha2017}.

In any case, in \cite{LinMelPug1977} it is proved that
$\operatorname{Lie}(2)=0$ and $\operatorname{Lie}(3)=1,$ and in
\cite{LiLli2012} that $\operatorname{Lie}(4)=1.$ In particular, the
proof of this last result is not easy at all. The first not known
number is $\operatorname{Lie}(5)\ge2.$

There is a classical tool, based on the construction of the
so-called Dulac functions that usually gives elegant  proofs of the
upper bound of the number of limit cycles. It is the
Bendixson--Dulac theorem. We state a particular version of it, which
is useful  in many cases to prove uniqueness (and hyperbolicity) of
limit cycles.

\begin{theo}[A very particular version of Bendixson--Dulac theorem, \cite{GasGia2006}]\label{th:bd} Let
$V:\R^2\to\R$ a $\mathcal{C}^1$    function such that $\nabla V$
vanishes on $\{V(x,y)=0\}$ at finitely many points  and the set
$\,\R^2\setminus \{V(x,y)=0\}$ has finitely many connected
components, all them are simply connected but eventually one, that
might have a hole (i.e., its fundamental group is $\mathbb{Z}$).
Assume there exists $s\in\R$ such that
\[ M_s=\frac{\partial V}{\partial x} P+\frac{\partial
V}{\partial y} Q+s \left(\frac{\partial P}{\partial
x}+\frac{\partial Q}{\partial y}\right)V
\]
does not change sign and vanishes only on a set of zero mesure  and,
moreover, that the set $\{V(x,y)=0\}$ does not contain periodic
orbits of \eqref{eq:bd}. Then, the $\mathcal{C}^1$ differential
system
\begin{equation}\label{eq:bd}
\dot x= P(x,y),\quad \dot y=Q(x,y),
\end{equation}
has:
\begin{itemize}
\item[(i)] not periodic orbits when either $s\ge0$ or no special region with a
hole exists,

\item[(ii)] at most one  periodic orbit  when $s<0$ that, when exists is a hyperbolic limit
cycle.
\end{itemize}
\end{theo}

The idea of its proof when $s\ne0$ is to show that the possible
periodic orbits can not cut the set $\{V(x,y)=0\}$ and later to
apply the Dulac theorem to each of the  connected components of
$\R^2\setminus \{V(x,y)=0\}$ with the Dulac function $|V|^{1/s}.$
When $s=0$ it is easier to be proved, simply observing that
$M_0=\dot V.$ To see a detailed proof of a  more general version of
the above result, and several examples of application, see
\cite{GasGia2013} and their references. We detail a couple of
examples of application for  systems of the form \eqref{eq:lieg}.

Consider system \eqref{eq:lieg} with $F(x)=cx^3+x^5.$ To prove that
it has at most 1 limit cycle we will apply Theorem \ref{th:bd} with
$F(x,y)=y^2-F(x)y+x^2+2c/5$ and $s=-1.$ Then
$M_{-1}(x,y)=2x^2(10x^4+10cx^2+3c^2)/5.$ It is easy to see that
$M_{-1}(x,y)\ge0$ and vanishes only on the line $x=0.$ Moreover,
since $F$ is quadratic on $y,$ $\R^2\setminus \{V(x,y)=0\}$ has at
most one connected component that can have a hole. Hence, the
corresponding system \eqref{eq:lieg} has at most one limit cycle
which  is hyperbolic when it exists. In fact, the divergence of the
vector field associated to the system is $3cx^2+5x^4.$ Since when
$c\ge0$ it is always greater or equal  that zero,   in this case the
system has no limit cycles by the classical Bendixson theorem. When
$c<0$ it is not difficult to see that the limit cycle exists. A
similar approach can be used to prove the uniqueness and
hyperbolicity of the limit cycle when $F(x)=cx^{2k+1}+x^{2m+1},$ for
any natural numbers $k<m,$ see \cite{GasGia2002}.

In \cite{SanCon1964} it is proved that taking
$F(x)=\frac{x(1-cx^2)}{1+cx^2}$ with $c>0,$ system \eqref{eq:lieg}
has at most 1 limit cycle and that a limit cycle exists for some
values of $c.$ By using Theorem \ref{th:bd} with $V(x,y)=
y^2-F(x)y+x^2$ and again $s=-1$ a simple and algebraic proof, by
using the above theorem,  of the uniqueness and hyperbolicity of the
limit cycles is given in \cite{GasGia2013}. The theorem applies
because
\[
M_{-1}(x,y)=-\frac {4cx^4}{(1+cx^2)^2}\le0
\]
and it only vanishes on the line $x=0,$ and $\R^2\setminus
\{V(x,y)=0\}=\R^2\setminus\{(0,0)\}$ has only one connected
component, which has a hole. This last assertion follows because the
discriminant of $V$ with respect to $y$ is
\[
\operatorname{dis}_y(V(x,y))=F^2(x)-4x^2=-{\frac {{x}^{2} \left(
c{x}^{2}+3 \right)  \left( 3\,c{x}^{2}+1
 \right) }{ \left( c{x}^{2}+1 \right) ^{2}}}\le0
\]
and only vanishes at $x=0.$

In next problem we propose to apply this method for low degree
Li\'{e}nard systems.

\caixa{
\begin{problem}
Find a proof using Dulac functions that
$\operatorname{Lie}(3)=1$ and $\operatorname{Lie}(4)=1$.
\end{problem}}

We remark that a proof that $\operatorname{Lie}(2)=0,$ follows
easily from the above theorem. When $F(x)=ax+bx^2,$ if we consider
$V(x,y)={\rm e}^{-2by}$ and $s=1,$ it holds that $M_1(x,y)=-a {\rm
e}^{-2by}.$ Hence if $a\ne0$ this Li\'{e}nard system has not periodic
orbits. When $a=0$ it has a reversible center at the origin, so it
has periodic orbits but not limit cycles.

\subsection{Riccati and Abel differential equations.}

Abel differential equations appear in the study of some planar
vector fields, see for instance Section \ref{se:rig}, but they are
also interesting by themselves. In general, they write as
\begin{equation}\label{eq:abel}
\frac{{\rm d} x}{{\rm d} t}=A_3(t)x^3+A_2(t)x^{2}+A_1(t)x+A_0(t),
\end{equation}
where all the functions $A_j$ are $\mathcal{C}^1$ and $T>0$
periodic. We are interested on finding conditions for these
functions to control their number of $T$ periodic solutions. Usually
the $T$ periodic solutions that are isolated among all the $T$
periodic solutions are also called limit cycles.

It is remarkable that while when $A_3=0$ (the Riccati differential
equation) the maximum number of limit cycles is two, there is no
upper bound for the number of limit cycles for general Abel
differential equations \eqref{eq:abel}, even when the functions
$A_j$ are trigonometrical polynomials, see \cite{Lin1980}.

The upper bound for Riccati  differential equation follows for
instance from the fact that, on its interval of definition, the
solution of this differential equation, when $A_3=0$ and satisfying
$\varphi(0;\rho)=\rho$ is
\[
x=\varphi(t;\rho)=\frac{B(t)\rho+C(t)}{D(t)\rho+E(t)},
\]
 where $B,C,D,E$ are smooth functions that
depend on $A_j,j=0,1,2,$ see for instance \cite{Hil1997}. Hence, for
each fixed $t$, it is a M\"{o}bius map. Therefore  its number of
periodic solutions is given by the number of solutions of the
quadratic equation obtained from the condition
$\varphi(T,\rho)=\rho.$ Moreover, limit cycles correspond to
isolated solutions of the quadratic equation. Nevertheless we only
know how to obtain explicitly these four functions when a particular
solution of the Riccati equation is known, see for instance
\cite{CimGasMan2006}. Hence the following problem remains:

\caixa{
\begin{problem}
For a general $T$ periodic Riccati differential equation
    \begin{equation*}
    \frac{{\rm d} x}{{\rm d} t}=A_2(t)x^{2}+A_1(t)x+A_0(t)
    \end{equation*}
give effective criteria to know when it has a continuum of periodic
solutions, or it has exactly 2, 1 or 0 limit cycles.
\end{problem}}

Two useful results to obtain upper bounds on  the number of limit
cycles for Abel differential equation \eqref{eq:abel} are:

\begin{itemize}

\item[(i)] If $A_3\ne0$ and does not change sign, then the maximum
number of limit cycles is~3, see \cite{GasLli1990}.

\item[(ii)]   If $A_0=A_1=0$ and there exist $a,b\in\R$ such that
$aA_3+bA_2\ne0$ and does not change sign, then the maximum number of
limit cycles is 3, see \cite{AlvGasGia2007}. Notice that one of them
is $x=0.$

\end{itemize}

One of the simplest natural open questions for Abel equations is:

\caixa{
\begin{problem}

Consider the family of trigonometric Abel differential equations
$$
\frac{{\rm d}x}{{\rm d}t}=(a_0+a_1\sin t+ a_2\cos t)x^3+(b_0+b_1\sin t+b_2\cos
t) x^2.
$$
Is  3  its maximum number of $2\pi$ periodic   limit cycles?
\end{problem}}

Notice that for the above differential equation $x=0$ is always a
periodic solution. So, if it is isolated from other $2\pi$ periodic
solutions, it is one of these limit cycles. In \cite{AlvGasGia2007}
the problem is introduced, the above two general results are applied
to this particular case, obtaining some particular positive answers,
and the existence of examples with at least 3 limit cycles is
established. In \cite{Bra2009} further complementary results are
obtained. In particular, it is proved that the answer is yes when
$a_0b_0=0.$

In fact, the above problem can be extended to next one:

\caixa{
\begin{problem}
Given two integer numbers $p>q\ge 2,$ and $m,n\in\N,$ find the
maximum number of $2\pi$ periodic limit cycles for next family of
Abel type differential equations
$$
\frac{{\rm d}x}{{\rm d}t}=A_m(t)x^p+B_n(t) x^q,
$$
where $A_m$ and $B_n$ are $2\pi$-trigonometric polynomials with
respective degrees $m$ and $n.$
\end{problem}}

This question was introduced in \cite{AlvGasYu2008}, where some
lower bounds of the number  of limit cycles were given. A recent
improvement of these bounds has been obtained in
\cite{HuaTorVil2020}.

In fact, the class of Abel type equations is very interesting and
intriguing. For instance, the following result proved in
\cite{GasGui2006} extends the results of previous item (i). We
remark  that the result when $n$ is odd was also obtained in
\cite{Pan1998}.

\begin{theo} Consider the $\mathcal{C}^1,$ $T$ periodic Abel type
differential equation
\begin{equation*}
\frac{{\rm d} x}{{\rm d} t}= A_3(t)x^3+A_2(t)x^{2}+A_1(t)x+A_0(t),
\end{equation*}
where $n\ge3$ and $0\ne A_n$ does not change sign. Then:
\begin{itemize}

\item[(i)] If $n$ is odd, it has at most 3 limit cycles and the upper bound is sharp.

\item[(ii)] If $n$ is even, there is no upper bound for its number of
limit cycles.
\end{itemize}
\end{theo}

\subsection{A new Hilbert XVIth type problem.}\label{se:mon}

For each $m\in\N$ fixed,  consider the following family of
polynomials differential equations:

\begin{itemize}
\item Family  $\mathcal{M}_m$ given by
\[
(\dot x,\dot y)=\sum _{j=1}^m a_j X_j(x,y),\quad \mbox{with}\quad
X_j(x,y)=\begin{cases}\big(x^{n_j}y^{k_j},0\big),&  or,\\
\big(0,x^{n_j}y^{k_j}\big),&
\end{cases}
\]
where $(a_1,a_2,\ldots,a_m)\in\R^{m}$ and the couples
$(n_j,k_j)\in\N^2$ vary among all the possible values. Varying $m,$
this family covers all polynomial differential equations.
\end{itemize}

The letter $\mathcal M$ is chosen because the important point is to
count the  number of involved monomials. We define
$\mathcal{H}^M[m]\in\N\cup\{\infty\}$ to be the maximum number of
limit cycles that systems of the family $\mathcal{M}_m$ can have.
This point of view is similar to the one of counting the number of
real solutions of planar fewnomial systems, see Section
\ref{se:few}. In the recent preprint \cite{BuzCarGas2020} we prove:

\begin{theo} It holds that $\mathcal{H}^M[m]=0$ for $m=1,2,3$ and for  $m\ge 4,$
$\mathcal{H}^M[m]\ge m-3.$ Moreover, there exists a sequence of
values of $m$ tending to infinity such that $\mathcal{H}^M[m]\ge
N(m),$ where
\[
N(m)=\Big(\frac{(\frac{m-3}2)\log(\frac{m-3}2)}{\log
2}\Big)(1+o(1)).
\]
\end{theo}

 The proof of the first part for $m=1,2,3$ follows by a case by case
 study. In particular, we prove that systems
$(\dot x,\dot y)= \big(ax^py^q, bx^iy^j+cx^ky^l\big),$ where
$(a,b,c)\in\R^3$ and $(p,q,i,j,k,l)\in\N_0^6,$ with
$\N_0=\N\cup\{0\},$ have no limit cycle. Currently we are working to
try to prove that $\mathcal{H}^M[4]=1$.

The fact that $\mathcal{H}^M[m]\ge m-3$  for $m>3$ is a
straightforward consequence of  known results about classical
Li\'{e}nard systems. In fact, an example with this number of limit
cycles is already provided by the one given in \cite{LinMelPug1977}
to prove that $\operatorname{Lie}(n)\ge[(n-1)/2],$ see Section
\ref{se:lie}, simply taking $F$ odd with $n=2m-5.$ Then, the Li\'{e}nard
system \eqref{eq:lieg} has $m$ monomials and $m-3$ limit cycles.

The second part is a direct corollary of the recent paper
\cite{AlvColMaePro2020}  where the authors study limit cycles for
generalized slow-fast Li\'{e}nard systems.

It is worth to mention that the celebrated examples of quadratic
systems that prove that $\mathcal{H}(2)\ge4$ are given by systems
with $m=8$ monomials and so they have $m-4$ limit cycles, see
\cite{CheArtLli2003,Per1984}. The slow-fast example of Li\'{e}nard
equation of degree 6 and 4 limit cycles, that gives a counterexample
of Lins, de~Melo and Pugh's conjecture, has also 8 monomials, see
\cite{MaeDum2011}. The cubic system given \cite{LiLiuYan2009} that
shows that $\mathcal{H}(3)\ge13$ has $m=9$ monomials and at least
$m+4$ limit cycles.

Under the light of the above results, some natural problems are:

%\newpage

\caixa{
\begin{problem}
(i) Find upper and lower bounds for $\mathcal{H}^M[m]$.

(ii) Find the minimal $m$  such that there exists a planar
polynomial differential system with $m$ monomials having at least
$m+1$ limit cycles: the simplest polynomial differential system with
more limit cycles than monomials.
\end{problem}}

\subsection{A second order differential equation} In
\cite{BraTor2010,Ure2016,Ure2017} there are several motivations to
study the $T$ periodic solutions of the second order singular, $T$
periodic differential equation $x^p(t) x''(t)=f(t),$ $0<p\in\R.$
From their results the following interesting problem can be
formulated:

\caixa{
\begin{problem} Let $f(t)$ be a continuous $T$ periodic function,
$f(t)\not\equiv 0$ and $0<p\in\R.$ Find necessary and sufficient
conditions on the function $f$ that ensure the existence of positive
$T$ periodic solutions of $x^p(t) x''(t)=f(t).$
\end{problem}}

Notice that a simple necessary condition is that $f$ changes sign,
because if $x(t)$ is a $T$ periodic positive solution, then
\[
\int_0^T \frac{f(t)}{x^p(t)}\,{\rm d}t=\int_0^T x''(t)\, {\rm
d}t=x'(T)-x'(0)=0.
\]
A second necessary condition is that $\int_0^T f(t)\,{\rm d}t<0.$
This can be easily proved by using integration by parts. If $x(t)$
is a $T$ periodic positive solution, then
\begin{align*}
\int_0^T f(t) \,{\rm d}t&= \int_0^T x^p(t)x''(t) \,{\rm d}t=
x^p(t)x'(t)\Big|_{t=0}^{t=T}-p\int_0^T x^{p-1}(t)(x'(t))^2 \,{\rm
d}t\\&=-p\int_0^T x^{p-1}(t)(x'(t))^2 \,{\rm d}t<0.
\end{align*}

In fact, in \cite{Ure2016} it is proved that if $p\ge2$ and $f(t)$
has  only nondegenerate zeroes, meaning  that it is continuously
differentiable, with nonvanishing derivative, in a neighbourhood of
each of its zeroes, the above two necessary conditions are also
sufficient. On the other hand, the same author proves in
\cite{Ure2017} that when $p=5/3$ there is a function $f,$ of class
$\mathcal{C}^\infty,$ satisfying both necessary conditions and such
that the corresponding differential equation has no positive
periodic solution.

\subsection{Number of centers} By Bezout's theorem, a planar polynomial differential systems of degree
$n>0,$ with finitely many equilibrium points, has at most $n^2$
equilibrium points. Moreover, at most $(n^2+n)/2$ can have index
$+1,$ see for instance \cite{CimLli1989,Kho1979}. Therefore, if we
define $\mathcal{C}_n$ as the maximum number of centers for this
class of polynomials systems, it holds that    $\mathcal{C}_n\le
(n^2+n)/2,$ because, remember that all centers have index $+1.$ It
is also clear that $\mathcal{C}_1=1.$

Moreover, by using the beautiful Euler--Jacobi's formula, it was
proved  in \cite{CimGasMan1993p} that for $n\ge2$  not all points of
index $+1$ can lie on the same algebraic curve of degree at most
$n-1.$ Since all centers are on the algebraic curve given by the
divergence of the vector field equal zero, which has at most degree
$n-1,$ we get that $\mathcal{C}_n\le (n^2+n)/2-1.$

Recall that this formula, for the  planar case, asserts that if a
polynomial system $P(x,y)=0,Q(x,y)=0,$ with $P$ and $Q$ with
respective degrees $n$ and $m,$ has exactly $nm$ solutions (hence
all them are finite and simple) then it holds that
\[
\sum_{\{(u,v)\,:\, P(u,v)=Q(u,v)=0\}} \frac
{R(u,v)}{\det(\operatorname{D}(P,Q))(u,v)}=0,
\]
for any polynomial $R(x,y)$ of degree smaller than $n+m-2,$ see for
instance \cite{GriHar1978}. Here, $\operatorname{D}(P,Q)$ denotes
the differential of the map $(P,Q).$

On the other hand, in \cite{CimGasMan1993} it is proved that planar
polynomial Hamiltonian differential systems of degree $n$ have at
most $[(n^2+1)/2]$ centers and that this upper bound is attained.
Hence  $[(n^2+1)/2]\le \mathcal{C}_n.$ In fact, examples of
Hamiltonian systems with this number of centers are not difficult to
be obtained. It suffices to consider
\[
\dot x =F(y),\quad \dot y=-F(x),\quad \mbox{with}\quad
F(u)=\prod_{j=1}^n (u-j).
\]
For these systems, centers and saddles are located like white and
black squares on an $n\times n$ chessboard.  In short, for $n\ge2$
it is known that
\[
\Big[\frac{n^2+1}2\Big]\le \mathcal{C}_n\le \frac{n^2+n}2-1.
\]
Hence $\mathcal{C}_2=2,$  $\mathcal{C}_3=5,$ and $8\le
\mathcal{C}_4\le 9.$

\caixa{
\begin{problem} Determine the maximum number of centers,
$\mathcal{C}_n,$ for planar polynomial differential systems of
degree $n\ge4.$
\end{problem}}

Once the number $\mathcal{C}_n$ is determined, it is also
interesting to know the different possible phase portraits that
systems having these maximal number of centers can have, see for
instance \cite{Blo2010}, where the Hamiltonian case when $n=3$ is
studied. One of the reasons is that these systems are good
candidates to have after perturbation different configurations with
many limit cycles.

\subsection{On some rational difference equations}

A classical problem for a given family of planar vector fields is
the so-called center-focus problem. It  consists on the distinction
between these two types of monodromic equilibrium points: center or
focus, or shortly into determining all centers of the family. Next
we present a similar question for some rational difference
equations. The goal here will be to determine all  difference
equations that are periodic. It is proved in \cite{CimGasMan2006}
that this periodicity is also strongly related with the complete
integrability of the discrete dynamical system associated to the
difference equation.

Let us introduce some definitions and state precisely the problem.
Consider the family of $k$-th order difference equations
\begin{equation}\label{eq:de}
x_{n+k}= \frac{A_0+A_1x_n+A_2x_{n+1}+\cdots+
A_kx_{n+k-1}}{B_0+B_1x_n+B_2x_{n+1}+\cdots+ B_kx_{n+k-1}},
\end{equation}
with $\sum_{i=0}^k A_i>0,$  $\sum_{i=0}^k B_i>0,$ $A_i\ge0,$
$B_i\ge0,$ and $A_1^2+B_1^2\ne0.$ For every initial condition
$(x_1,x_2,\ldots,x_k)\in(0,\infty)^n,$ they define a sequence
$\{x_i\}_{i\ge1}$ of positive real numbers. One of these difference
equations is called $p$-periodic if for all these initial conditions
it holds that $x_n=x_{n+p},$ for all $1\le n\in\N,$ and this value
$0<p\in\N$ is the smallest number with this property. That is, all
the sequences with positive initial conditions generated by
\eqref{eq:de} are $p$-periodic.

The following examples of $p$-periodic difference equations of the
form \eqref{eq:de} are known:
\begin{align}\label{eq:exa}
&x_{n+1}=x_n\quad\mbox{with}\quad  p=1, \quad\quad
x_{n+1}=\frac1{x_n} \quad\mbox{with}\quad p=2,\nonumber \\
&x_{n+2}=\frac{x_{n+1}}{x_n} \quad\mbox{with}\quad p=6,\quad
x_{n+2}=\frac{1+x_{n+1}}{x_n} \quad\mbox{with}\quad p=5,\\
& x_{n+3}=\frac{1+x_{n+1}+x_{n+2}}{x_n} \quad\mbox{with}\quad
p=8.\nonumber
\end{align}
Moreover, every $p$-periodic $k$-th order difference equation
produces in a natural way, and  for each $\ell\in\N,$ another one
which is $p\ell$-periodic and of $k\ell$-th order. For instance, the
one of second order given in \eqref{eq:exa} gives
\[
x_{n+2\ell}=\frac{x_{n+\ell}}{x_n} \quad\mbox{with}\quad
p=6\ell,\quad x_{n+2\ell}=\frac{1+x_{n+\ell}}{x_n}
\quad\mbox{with}\quad p=5\ell.
\]
Similarly, every $p$-periodic $k$-th order difference equation can
be unfold into a 1-parametric family with the same property. It
suffices to consider for any $n\in\N,$ $y_n=ax_n,$ with $0\ne
a\in\R.$ For instance, the above ones  give rise to
\[
y_{n+2\ell}=\frac{a y_{n+\ell}}{y_n} \quad\mbox{with}\quad
p=6\ell,\quad y_{n+2\ell}=\frac{a^2+a y_{n+\ell}}{y_n}
\quad\mbox{with}\quad p=5\ell.
\]
For short, all these new difference equations are called equivalent
to \eqref{eq:exa}.

\caixa{
\begin{problem} Are there rational difference equations of the form
\eqref{eq:de} that are not equivalent to the five ones given in
\eqref{eq:exa}?
\end{problem}
}

 The answer to the above question for $k\in\{1,2,3,4,5,7,9,11\}$ is
no, see \cite{CimGasMan2004}. We remark that when the condition of
non-negativeness of the coefficients  of \eqref{eq:de} is removed
much more periods and periodic difference equations appear. For
instance, when $k=1$ there are periodic M\"{o}bius maps with all the
periods. To see more information about related problems, see
\cite{CimGasManMan2016} and its references.

\section{Period function}\label{se:pf}

Let $\gamma(s),$ with $s$ in a real open  interval, be a smooth
parameterized continua of periodic orbits of a smooth planar
autonomous differential system. Usually, if the system is
Hamiltonian the parameter $s$ is taken to be the energy of the
system. When the continuum of periodic orbits ends in a critical
point, then the maximal set covered by them is called period annulus
of the point. The function that assigns to each $s$ the minimal
period of $\gamma(s)$ is called period function and it is usually
denoted by $T(s).$ The zeroes of $T'(s)$ are called critical periods
and determine them is a key point to know the behaviour of $T(s).$
To know properties of this function (monotonicity, number of
oscillations,\ldots) is interesting from a theoretical point of
view, as well as for applications for instance in physics or ecology
(\cite{ConVil2008,Rot1985,Wal1986}).

\subsection{A class of Hamiltonian systems.}

From the results of \cite{Col1996,Cop1993,Gas1997} it is known that,
on the period annulus of the origin, the period function has at most
one critical period for the family of Hamiltonian systems with
Hamiltonian
\[
H(x,y)= \frac12\big(x^2+y^2\big)+H_m(x,y),
\]
where $H_m(x,y)$ is  a  homogeneous polynomial of degree $m\ge3.$
Next question proposes to study if the same result holds for a more
general class of Hamiltonian systems.

\caixa{
\begin{problem} Consider a  Hamiltonian system with a center at the
origin and Hamiltonian
\[
H(x,y)= H_{2n}(x,y)+H_m(x,y), \qquad m>2n,\] where $H_{2n}$ and
$H_m$ are homogeneous  polynomials of degrees $2n$ and $m,$
respectively. Has the period annulus of the origin at most 1
critical period?
\end{problem}}

In \cite{Alv2017} it is proved that   the answer is  yes when $m\ge
4n-2$. So it remains to study the cases $2n<m<4n-2.$ Notice that the
simplest open question corresponds to the Hamiltonian $H(x,y)=
H_{4}(x,y)+H_5(x,y).$

\subsection{Systems with homogeneous components.}\label{se:shc}

We consider again systems
$$
\begin{cases}
\dot x=P_{2k+1}(x,y),\\
\dot y=Q_{2\ell+1}(x,y),
\end{cases}
$$
where $P_{2k+1}$ and $Q_{2\ell+1}$ are  homogeneous
polynomials of degrees $2k+1$ and $2\ell+1,$ respectively.

\caixa{
\begin{problem}
(i) Characterize the centers of the above family.

(ii) Which is the maximum number of oscillations of the period
function for the centers of the above family?
\end{problem}}

When $k=\ell$ both questions have a simple answer. The centers can
be characterized studying their expression in polar coordinates
$(x,y)=(r\cos\theta,r\sin\theta),$ because they are easily
integrable, see \cite{Arg1968}. In fact, they write as
\[
\dot r=f(\theta)r^{2k+1},\quad \dot\theta=g(\theta)r^{2k},
\]
where
\begin{align*}
f(\theta)&=P_{2k+1}(\cos\theta,\sin\theta)\cos\theta+Q_{2k+1}(\cos\theta,\sin\theta)\sin\theta,\\
g(\theta)&=Q_{2k+1}(\cos\theta,\sin\theta)\sin\theta-P_{2k+1}(\cos\theta,\sin\theta)\cos\theta.
\end{align*}
Hence the center conditions are that $g$ does not vanish (otherwise
the system would have invariant lines through the origin) and
\[
\int_0^{2\pi} \frac{f(\theta)}{g(\theta)}\,{\rm d}\theta=0,
\]
where we have obtained this last equality because the solution of
$\frac{dr}{d\theta}= \frac{f(\theta)}{g(\theta)}r$ with initial
condition $r(0)=s>0$ is
\[
r(\theta;s)= s \exp\left( \int_0^{\theta}
\frac{f(\psi)}{g(\psi)}\,{\rm d}\psi  \right).
\]
Hence by imposing that $r(0)=r(2\pi)$ the condition follows.

The period function can also be obtained from the above equations.
In fact,
\[
T(s)=\int_0^{2\pi} \frac 1{|g(\theta)|r^{2k}(\theta;s)}\,{\rm
d}\theta= \Bigg( \int_0^{2\pi} \frac {\exp\left( -2k\int_0^{\theta}
\frac{f(\psi)}{g(\psi)}\,{\rm d}\psi\right)}{|g(\theta)|}\,{\rm
d}\theta \Bigg) \frac 1{s^{2k}}=: \frac {A_k}{s^{2k}},
\]
and it is constant for  $k=0$ (these are the linear centers) and
decreasing for $k\ge1.$

In \cite{CaiLli2000} the authors studied all the phase portraits
when $k=0$ and $\ell=1,$ modulus the number of limit cycles, which
recall that at least is 2, see Section \ref{se:hom}.

\subsection{About the maximum number of critical
periods}\label{se:cp}

Let  $\mathcal{H}(n)$ denote the maximum number of  limit cycles
that planar
 polynomial systems  of degree  $n$ can  have.
 From \cite{ChrLlo1995} it is
known that
\[
\mathcal{H}(n)\ge K n^2 \log (n),\quad
\mbox{for some}\quad  K>0.\]  On the other hand, if
we denote as $\mathcal{T}(n)$ the maximum number of
 critical periods  that planar  polynomial
systems  of degree  $n$ can have, from the results of
\cite{GasLiuYang2010} it is  also known that
\[
\mathcal{T}(n)\ge \frac14 n^2.
\]
Recently, this lower bound has been essentially doubled in
\cite{Cen2021,MaeWyn2020}. A natural question is:

\caixa{
\begin{problem}
Is it true that $\mathcal{T}(n)\ge C n^2\log(n)$
for some $C>0$?
\end{problem}}

\subsection{Reversible quadratic systems}\label{se:rqs} Although there are same
subsequent results, in \cite{MarMarVil2006}  there is an excellent
source of information about one of the most famous open problems
about critical periods. It was proposed by Chicone in 1994 in a
review of MathSciNet and it reads as follows:

\caixa{
\begin{problem}
Consider the family of reversible quadratic centers
   \begin{equation}\label{eq:rqs}
    \begin{cases}
    \dot x=-y+xy,\\
    \dot y=x+Dx^2+Fy^2.
    \end{cases}
    \end{equation}
Is 2 its maximum number of critical periods?
\end{problem}}

The above systems are sometimes called Loud's systems, because this
author studied them in 1964, see \cite{ChaSab1999,MarMarVil2006}.

\subsection{Some  reversible equivariant planar differential systems}\label{se:equi}

Any planar analytic system, $(\dot x, \dot y)=(f(x,y),g(x,y)),$ can
be written in complex coordinates as $\dot z=F(z,\bar z),$ where
$z=x+{\rm i}y.$ Moreover, when the origin is a weak focus, after a
constant rescaling of time, it writes as $\dot z={\rm i}z+G(z,\bar
z),$ where $G$ starts at least with second order terms.

Recall that, in real coordinates, one of the simplest criteria to
know that the origin  is a center is the so-called Poincar\'{e}'s
reversibility criterion. It simply says that if a equilibrium point
(the origin) is monodromic and the system is invariant by the change
of variables and time $(x,y,t)\to (x,-y,-t)$ then it  is a center.
This is so because if $(x(t),y(t))$ is a solution of the system,
then the same happens for $(x(-t),-y(-t))$ and by the monodromy
condition and the uniqueness of solutions, both trajectories
intersect and, hence, they coincide. This proves that this solution
is a periodic orbit which is symmetric with respect the line $y=0.$
Notice that in complex variables this criterion works when the
differential equation is invariant by the change of variables and
time $(z,t)\to (\bar z,-t).$ Simply by considering symmetries with
respect arbitrary straight lines passing through the origin we
obtain the following well-known general result: if the origin  of a
differential equation $\dot z=F(z,\bar z)$ is a monodromic critical
point and, for some $\alpha\in\R,$ this equation is invariant by the
change of variables and time $(z,t)\to({\rm e}^{{\rm i}\alpha} \bar
z,-t),$ then the origin is a (reversible) center.

If we consider the origin of $\dot z=F(z,\bar z)$ to be a weak focus,  and
we write this equation as
\[
\dot z={\rm i}z+\sum_{m+n\ge 2}  A_{m,n}z^m\bar z^n,\quad A_{m,n}\in\C,
\]
then the condition for this equilibrium point to be a reversible
center is simply that there exists some $\alpha\in\R$ such that
$A_{m,n}=-\bar {A}_{m,n} {\rm e}^{{\rm i}(1-m+n)\alpha},$ for all
$m,n\in\N.$

Another  remarkable class of planar systems are the so-called
$\Z_k$-equivariant differential equations, see for instance
\cite[Sec. 7]{Li2003} and their references. They are differential
equations $\dot z=F(z,\bar z)$ that are invariant by a rotation
through $2\pi/k$  about the origin, or in other words, such that the
change of variable $z\to {\rm e}^{{\rm i}\beta}z,$ for
$\beta=2\pi/k, k\in\N,$ leaves them invariant. For these
differential equations, the phase portrait on each sector centered
at the origin and width  $2\pi/k,$  is repeated $k$ times.

Consider the following family of polynomial  $\Z_k$-equivariant differential equations
\begin{equation}\label{eq:equi}
\dot z= {\rm i} z+(z\bar z)^n z^{k+1},
\end{equation}
with $n\in\N$ and $k$ a positive integer. It has  a reversible
center at the origin because of Poincar\'{e}'s extended result with
$\alpha=\pi/k.$ We are interested on the behaviour of the period
function associated to this center. Notice that when $n=0$ the
differential equation is holomorphic and so it has an isochronous
center at the origin (\cite{ChaSab1999}).

\begin{prop} Consider the period function associated to the origin for \eqref{eq:equi}.
Then its behaviour and number of critical periods coincide with the
one of the period function of the origin of the quadratic reversible
center \eqref{eq:rqs} with $F=1+D$ and
$D={-k}/{(2(k+n))}\in\left[-1/ 2,0\right).$
\end{prop}

\begin{proof} Equation \eqref{eq:equi} in polar coordinates $z=r{\rm
e}^{{\rm i}\theta}$ writes as
\[
\dot r=  r^{2n+k+1}\cos(k\theta),\quad \dot \theta=1 +
r^{2n+k}\sin(k\theta).
\]
Taking $R=r^{2n+k},$  $\Theta=k\theta,$ and reparametrizing the time
by the constant factor $k,$ the above system of equations is
converted into
\[
R'= b R^2\cos\Theta,\quad \theta'=1 + R\sin\Theta,
\quad\mbox{where}\quad b=1+\frac{2n}k.
\]
Introducing again real coordinates $X+{\rm i}Y=R{\rm e}^{{\rm
i}\Theta}$ this last system of equations writes as
\[
X'=-Y+bX^2-Y^2,\quad Y'=X+(1+b)XY.
\]
Now, taking $x=-(1+b)Y,$ $y=-(1+b)X,$ and a change of sign of the
time, we arrive to
 \begin{equation*}
    \begin{cases}
    \dot x=-y+xy,\\
    \dot y=x-\frac{1}{1+b}x^2+\frac b{1+b}y^2,
    \end{cases}
    \end{equation*}
which is precisely a system of the form \eqref{eq:rqs} with $F=1+D$
and $D=-1/{(1+b)}\in\left[-1/2,0\right).$

Finally, notice that the period of a periodic orbit surrounding the
origin for this last system is proportional to the time spend by a
periodic orbit of system \eqref{eq:equi} for going from $\theta=0$
to $\theta=2\pi/k.$ Since the system is $\Z_k$-equivariant, the
total period of this periodic orbit is $k$ times this last time, and
the result follows.
\end{proof}

Notice that when $n=0$ we recover one of the quadratic isochronous
centers, $(D,F)=(-1/2,1/2),$ see \cite{ChaSab1999}. Hence we have
reduced our problem to a similar one, but for the quadratic
reversible centers \eqref{eq:rqs} on the line $D-F+1=0$ and
$-1/2<D<0.$ Unfortunately, despite all the efforts done to study
this quadratic family, the behaviour of the period function on this
line is not yet know, although it is believed that it is monotonous
decreasing, see \cite{MarMarVil2006}.  Hence the following question
arises:

\caixa{
\begin{problem} Is the period function associated to the period annulus of the
origin of the differential equation $\dot z= {\rm i} z+ (z\bar z)^n
z^{k+1},$ with $n$ and $k$ a positive integers, monotonous
decreasing?
\end{problem}}

In fact, the above differential equation has other centers whose
period functions also deserve to be studied.

\section{Piecewise linear systems}\label{se:pls}

In non-smooth dynamics the differential equations appearing in the
simplest models are piecewise linear. Moreover, the discontinuity
curve is often given by a straight line.  These models have
attracted the attention of many scientists not only because its
simplicity, but also for the accuracy of the results obtained by
using them, compared  with the real observations, see more details
for instance  in  \cite{Aca2011, Bro2017, Kun2000}. We present a
couple of questions concerning their number of limit cycles.

\subsection{Algebraic limit cycles and related  questions.}

In this section we  compare some results about limit cycles for
quadratic systems with similar ones for piecewise linear systems
with a straight line of separation to highlight the parallelism
between both settings.

The existence of examples with 4 limit cycles for quadratic systems
has been already revealed in Section \ref{se:mon} and an example
with 3 limit cycles for piecewise linear systems is given in
\cite{LliPon2012}. We remark that when we consider limit cycles for
piecewise linear systems we only refer  to  crossing limit cycles, see
\cite{Ber2008,Fil1988}. This means that the two sides of the limit
cycle cut transversally the line of discontinuity and at these
points of cutting both vector fields point to the same half-plane,
see Figure \ref{fi:lc}. In other words, the crossing limit cycles
never follow the discontinuity line, avoiding the so-called sliding
motion, see again Figure \ref{fi:lc}.

Recall that a limit cycle of a smooth differential system is called
algebraic it it is an oval of an irreducible algebraic curve. The
degree of the limit cycle is the one of the curve. Similarly, a
(piecewise crossing) algebraic limit cycle for a piecewise linear
system is given also by a topological oval such that all its points
are contained,  on each of the sides of the discontinuity,  in an
irreducible algebraic curve in any of the two sides. Then the degree
of this limit cycle is a couple $(m,n)\in\N^2$ being each one of
these numbers the degrees of each one of the invariant algebraic
curves.

\begin{figure}[h]
    \centering\includegraphics{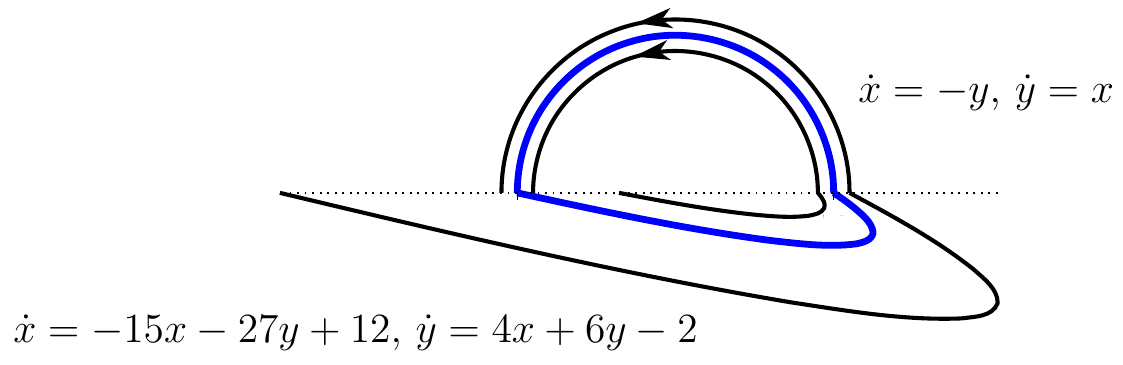}\caption{Example of algebraic (crossing) limit cycle of a piecewise
        linear system.}\label{fi:lc}
\end{figure}

In Figure \ref{fi:lc} an example of piecewise algebraic limit cycle of
degree $(2,2)$ for a piecewise linear system is showed, see
\cite{BuzGasTor2018}. It is formed by a piece of the parabola
$(x+3y)^2-3y-1=0$ and a piece of the circumference $x^2+y^2-1=0.$ We
remark, that although algebraic limit cycles with sliding do exist,
see for instance \cite{BuzGasTor2018}, we do not consider them.

The results about algebraic limit cycles for quadratic systems have
been obtained along the years in several papers, see for instance
\cite{ChrLliSwi2005} and the references therein. Until today,
irreducible algebraic limit cycles of degrees 2,4,5 and 6 are known
and it is also proved that there are not limit cycles of degree 3.
It is not known if these degrees are the only possible ones.

 On the other hand,  in \cite{BuzGasTor2018} it is proved that for piecewise linear
 systems, given any couple of natural numbers $(m,n)\in\N^2,$ with $n\ge2$ and
$m\ge2,$ there are algebraic limit cycles of degree $(m,n).$
Moreover, when $n=m$ it is also known that the two pieces of the
piecewise algebraic limit cycle do not correspond to a single
algebraic curve. In that paper also appear examples with  2
algebraic and hyperbolic limit cycles and examples with a  double semi-stable
algebraic  limit cycle.

Examples of quadratic systems with 2 critical periods are given for
the quadratic reversible centers \eqref{eq:rqs} introduced in
Section \ref{se:rqs}. See also \cite{GraVil2010} for other examples.
We do not know that the number of critical periods with piecewise
linear systems has been studied when the separation curve is a
straight line. The case where this curve is more general is studied
in \cite{WanYan2020}.

The table in next problem collects the above results and shows the
best known lower bound for the objects described in the left column.
When a question mark appears it means that it is not proved that the
given lower bound is the actual value. It is believed that the table
should be as it is but without question marks, but as we have
already explained, the only known result is that algebraic and
non-algebraic limit cycles never coexist  for piecewise linear
systems with a straight line of separation (\cite{BuzGasTor2018}).

\caixa{
\begin{problem} Improve this table:
\begin{center}
\begin{tabular}{||c|c|c|}
\hline
& Quadratic sys. & Piecewise linear sys.\\
\hline \hline  Limit cycles (l.c.)  & 4? & 3? \\\hline Algebraic
 limit cycles& 1?& 2?
\\\hline Non hyperbolic algebraic l.c.  & 0?&
1?
\\\hline Coex. of algebraic and non-algebraic  l.c.& NO?&
NO
\\
\hline Critical periods& 2?&?\\
\hline
\end{tabular}
\end{center}
\end{problem}}

\vspace{0.2cm}

\subsection{Another Hilbert's XVI type problem.}

Let $\mathcal{L}(n)$ denote the maximum number of (crossing) limit
cycles of planar piecewise linear differential systems with two
zones separated by a branch of an algebraic curve of degree $n$.  A
branch is an unbounded curve diffeomorphic to $\R$ and that defines
a closed set. We also stress that although the commonly used name is
linear, indeed both vector fields are affine. In principle, although
it seems improvable, we admit that some of the numbers
$\mathcal{L}(n)$ could be infinity.
 Recall that $\mathcal{H}(n)$ denotes the maximum number of  limit cycles
that planar polynomial systems of degree  $n$ can  have. With these
notations in mind we propose the following problem.

\caixa{
\begin{problem}
Improve, if possible, the lower bounds of this table:
\begin{center}
\begin{tabular}{||c|c||}
\hline
Polynomial case& Linear piecewise case\\
\hline \hline $\mathcal{H}(2)\ge4$&
$\mathcal{L}(1)\ge3$\\
\hline $\mathcal{H}(3)\ge13$&
$\mathcal{L}(2)\ge4$\\
\hline $\mathcal{H}(n)\ge K n^2 \log (n)$&
$\mathcal{L}(n)\ge[n/2]$\\
\hline
\end{tabular}
\end{center}
\end{problem}}

\vspace{0.2cm}

The lower bounds for the values of $\mathcal{H}(n)$ given in the
above table have already appeared in this paper, see Sections
\ref{se:mon} and \ref{se:cp}.  Recently, very good lower bounds for
$\mathcal{H}(n)$ and $n$ small are given in \cite{ProTor2019}. For
instance, $\mathcal{H}(4)\ge 28$ or $\mathcal{H}(5)\ge 37.$

The lower bounds for $\mathcal{L}(n)$ and $n=1,2$  are given in
\cite{LliPon2012} and \cite{GasTorZha2020}, respectively. In
\cite{BasBuzLliNov2019} it is proved that $\mathcal{L}(3)\ge 7$ and
in the recent preprint \cite{And2020} that $\mathcal{L}(3)\ge 8.$
Also in \cite{GasTorZha2020} the general lower bound for
$\mathcal{L}(n)$ given above is proved with the aim of showing that
$\mathcal{L}(n)$ tends to infinity when $n$ does. We believe that
there is room for improving it. Next, we include an idea of its
proof.

Define $f_n(x,\varepsilon)=\varepsilon T_n(x)$ with $\varepsilon>0$
suitable  small, where $T_n(x)$ is the  Chebyshev
    polynomial  of the first kind, i.e. for $|x|\le1,$
$T_n(x)=\cos(n\arccos x),$ and for $|x|>1,$ its analytic extension.
It is known that $T_n(x)$ is a polynomial of degree $n$ and all its
roots are real and in $[-1,1].$

The curves of degree $n,$
    $y=f_n(x,\varepsilon)$, have a single  branch and separate the  plane in two zones,
$\Omega^+$ when $y\ge f_n(x,\varepsilon)$ and $\Omega^-$ when $y\le
f_n(x,\varepsilon)$. We consider  the piecewise linear differential
systems
\begin{equation}\label{eq:2}
(\dot x,\dot y)=\begin{cases}
(x-4y-2,\frac 12 x-y),\quad &\mbox{on}\quad \Omega^+,\\
(-y+1,x),\quad &\mbox{on}\quad \Omega^-.
\end{cases}
\end{equation}
Some easy calculations show that the first integrals of each one of
the linear systems are
\[
H^+(x,y)=8y+x^2-4xy+8y^2\quad \text{and} \quad H^-(x,y)=-2y+x^2+y^2,
\]
respectively. If instead of the separation curve
$y=f_n(x,\varepsilon)$  we consider the separation line $y=0$, all
the solutions are   periodic orbits because $H^{\pm}(x,0)=x^2.$ The
important point is that the ones that pass trough the points $(\pm
x_k,0),$ where $x_k\ne0$ is a zero of $T_n(x),$ are the ones that
remain as  limit cycles,  for $\varepsilon$ small enough, see
Figure~\ref{fi:che}.

More specifically, let $m=[(n-2)/2]$, and $\pm x_{0}, \ldots, \pm
x_{m}$ be the $2(m+1)$ zeros of $f_n(x,\varepsilon)$ which are not
zero. Then
\[
x_k=\cos\left(\frac{2 k+1}{2n}\pi\right), \qquad k=0,1,\ldots,m.
\]
Then, for each $k\in \{0,1,\ldots,m\}$ and $\varepsilon$ small
enough,
$$
\Gamma_k:=\{(x,y)|\ H^+(x,y)=H^+(P_{\pm k}),\, y\ge 0\} \,\cup
\{(x,y)|\ H^-(x,y)=H^-(P_{\pm k}),\, y\le 0\},
$$
is a periodic orbit of our piecewise system \eqref{eq:2}, where $P_{\pm
k}=(\pm x_{ k},0)$ for $k\in\{0,1,\ldots,m\}$.

\begin{figure}[h]
\centering\includegraphics{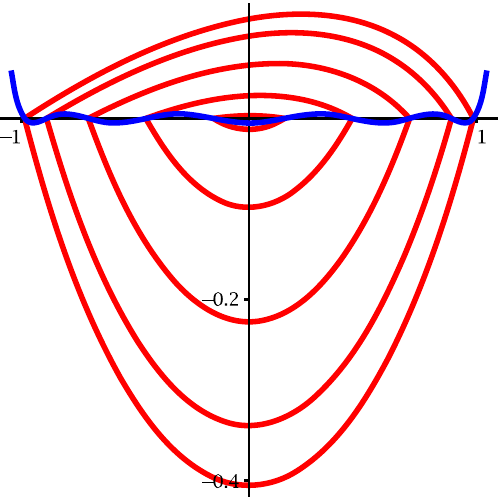} \caption{Separation curve defined
by a Chebyshev polynomial of degree 10 and 5 limit
cycles.}\label{fi:che}
\end{figure}

To prove that $\Gamma_k$ is a  hyperbolic limit cycle we compute the
derivative of the Poincar\'e map, which is a composition of two maps
and prove that it is not 1. This can be done by using  the nice
formula (\cite{AndVitKha1966})
\begin{equation*}
\Pi'(0)=\frac{ \langle X(0),(\gamma_0'(0))^\perp\rangle }{ \langle
X(T),(\gamma_1'(0))^\perp\rangle}
\exp\left(\int_0^T\operatorname{div} X(\varphi(t))\,{\rm d}t\right),
\end{equation*}
where $\langle \cdot,\cdot \rangle$ is the inner product of two
vectors, the superscript $\perp$ denotes the orthogonal of a two
dimensional vector, that is $(u,v)^\perp=(-v,u),$ and $\gamma_0(s)$
and $\gamma_1(s)$ are the local expressions of two transversal
sections $\Sigma_0$ and $\Sigma_1$ to a vector field $X$, and $T$ is
the time moving from $P=\gamma_0(0)$ to $\Pi(P)=\gamma_1(0),$ see
Figure \ref{fi:2}.

\begin{figure}[h]
    \centering\includegraphics{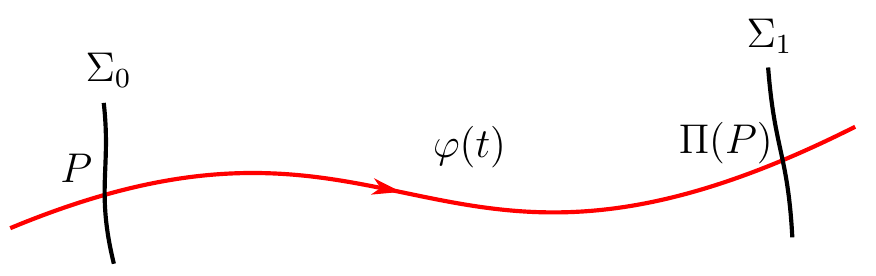}\caption{Poincar\'e map.}\label{fi:2}
\end{figure}

\section{On global asymptotic stability}\label{se:my}

We will present three problems about global asymptotic stability of
dynamical systems.

\subsection{A Markus-Yamabe problem for differential equations}

Markus-Yamabe conjecture for ordinary differential equations was
stated by L.~Markus and H.~Yamabe in 1960, and it said:

\begin{quote}
 Let $\dot x=F(x),$ $x\in\mathbb{R}^n,$ be a smooth
differential equation  such that $F(0)=0$ and  for all $x\in
\mathbb{R}^n,\,$  all the eigenvalues of  the matrix
$\mathrm{D}F(x)$ have   negative real part.  Then the origin is
globally asymptotically stable.
\end{quote}

Nowadays it is known that it is  true in dimensions 1 (it has a very
simple proof)  and~2 (see \cite{Fes1995,Glu1995,Gut1995}), and it is
false in higher dimensions. A smooth counterexample in dimension 4
was given in 1996 in \cite{BerLli1996}. It was obtained by
perturbing a continuous linear piecewise differential system defined in $\R^4$
having a hyperbolic periodic orbit and  so it also has a periodic
orbit. Polynomial counterexamples in dimension 3 and higher  were
given in 1997 in \cite{Cim1997}. These counterexamples are
\[
\dot x= -x+z_1(x+yz_1)^2,\quad \dot y=-y-(x+yz_1)^2,\quad \dot
z_i=-z_i,
\]
$i=1,2,\ldots, n-2.$ It can be seen that at any point all the
eigenvalues of the differential matrix $\mathrm{D}F$  are $-1,$ and that it has the
particular solution
\[(x,y,z_1,\ldots,z_{n-2})=\big(18 {\rm e}^t,-12 {\rm e}^{2t}, {\rm
e}^{-t},\ldots, {\rm e}^{-t}\big).\] Notice  that this solution
tends to infinity when $t$ increases. Therefore,  it appears the
following natural question:

\caixa{
\begin{problem}
Are there smooth vector fields in $\mathbb{R}^3$ under the
hypotheses of the Markus--Yamabe's conjecture and  having
periodic orbits?
\end{problem}}

\subsection{A Markus--Yamabe/La Salle  problem for DDS}

J.~P.~La Salle in 1976  proposed some  possible sufficient
conditions  for discrete dynamical systems with a fixed point, $
x_{m+1}=F(x_m),\,x\in \mathbb{R}^n, $ to be globally asymptotically
stable, see \cite{Las1976}. Some of them are discrete versions of
the Markus--Yamabe conditions given in previous section and have been
studied in \cite{CimGasMan1999,CimGasMan2001}. Two of these
conditions are:
\begin{align*} &\mbox{(C$_1$) For all }x\in
\mathbb{R}^n,\, \rho\left(\mathrm{D}F(x)\right)<1,\\
&\mbox{(C$_2$) For all }x\in \mathbb{R}^n,\,
\rho\left(\left|\mathrm{D}F(x)\right|\right)<1,
\end{align*}
where $\rho$ is the spectral radius of the differential matrix and,
given a square matrix $A=\big(a_{i,j}\big)$,  $|A|$ denotes the new
matrix with all its entries $|a_{i,j}|.$ It is known that
$\rho(A)\le \rho(|A|),$ see \cite{Gan1959}. Hence Condition~C$_2$ is
stronger than~C$_1,$ because $\rho(|A|)<1 \Longrightarrow \rho(A)<1,$
but it can be easily seen that $\rho(A)<1\,\, \not\!\!\Longrightarrow
\rho(|A|)<1.$

In fact, in \cite{CimGasMan1999} the authors already gave an example
due to W. Szlenk of a rational map satisfying condition~C$_1,$ for
$n=2,$ for which the origin is not globally asymptotically stable,
because it has a 4-periodic orbit. Similarly, polynomial maps
satisfying condition~C$_1,$ for $n\ge 3,$ for which the origin is
not globally asymptotically stable are presented in \cite{Cim1997}.
These maps have orbits that tend to infinity. The following question
remains open:

\caixa{
\begin{problem} Let $F:\mathbb{R}^2\rightarrow \mathbb{R}^2$ be a smooth map having a fixed point and
such that \[ \rho\left(\left|\mathrm{D}F(x)\right|\right)<1,\quad
\mbox{for all}\quad x\in\mathbb{R}^2.
\]
Is it true that the fixed point is globally asymptotically stable?
\end{problem}}

\subsection{Random linear differential or difference equations}

Given a $n$-th order linear homogeneous differential equation it is
natural to wonder which is the probability of the zero solution of
being a global attractor. Let us formalize this question.

Consider for instance the $3$-rd order linear differential equation
\begin{equation*}
A x'''(t)+Bx''(t)+Cx'(t)+Dx(t)=0,
\end{equation*} where  $A,B,C,D$ are real continuous random variables. It is natural to
require that all these  random variables are  independent and
identically distributed (i.i.d.\!). Also it seems reasonable to
impose that they are such that the random vector $(A,B,C,D)$ has
uniform distribution in $\mathbb{R}^4$. But such a  distribution is
impossible for unbounded probability spaces. Anyway, let us see that
there is a natural election for it.

It is clear that the solutions of the above differential equation do
not vary if we multiply the  equation by a positive constant. This
means that in the space of parameters, $\R^4,$ all the differential
equations with parameters belonging to the same half-straight line
passing through the origin are the same. Hence, we can ask for  a
probability distribution density~$f$ of the coefficients such that
the random vector
\begin{equation}\label{e:randomvector}
\left(\frac{A}{S}\,,\frac{B}{S}\,,\frac{C}{S}\,,\frac{D}{S}\right),
\quad\mbox{with}\quad S=\sqrt{A^2+B^2+C^2+D^2},
\end{equation}
has a uniform distribution on the sphere ${\mathbb S}^3 \subset
\R^4,$ that is a compact set. In \cite{CimGasMan2021} it is proved
the following result:

\begin{theo}
Let $X_1,X_2,\ldots,X_n$ be i.i.d.\! one-dimensional real random
variables with a continuous positive density function $f$. The
random vector
$$
\left(\dfrac{X_1}{S}, \dfrac{X_2}{S},\ldots,\dfrac{X_n}{S}\right),
\quad\mbox{with}\quad S=\Big(\sum_{i=1}^{n} X_i^2\Big)^{1/2},
$$
has a uniform distribution  in $\mathbb{S}^{n-1}\subset\mathbb{R}^n$
if and only if each $X_i$ is a normal random variable with zero
mean.
\end{theo}

Hence it is natural to  consider linear random homogeneous
differential equations of order $n$
\begin{equation}\label{eqdiflineal}
A_n x^{(n)}(t)+ A_{n-1} x^{(n-1)}(t)+ \cdots + A_2 x''(t)+A_1
x'(t)+A_0 x(t)=0,\end{equation} where all $A_j$ are
 i.i.d.\ random variables with $N(0,1)$ distribution.

To know the probability that the zero solution is globally
asymptotically stable is equivalent to know the probability, say
$p_n,$ that all the roots of its associated random characteristic
polynomial
$$
Q(\lambda)=A_n\lambda^{n}+A_{n-1}\lambda^{n-1}+\cdots
+A_1\lambda+A_0$$ have negative real part. Recall that the
conditions among the coefficients that imply this property are given
by algebraic relations among them and can be obtained via the
Routh--Hurwitz criterion.

In \cite{CimGasMan2021} it is proved that $p_n\le1/2^n,$ so
$\lim_{n\to\infty}p_n=0.$ Furthermore, it is also shown that
$p_1=1/2,$ $p_2=1/4,$ $p_3=1/16,$ $p_4<1/32,$ and, by using Monte
Carlo simulation, that $p_4\simeq0.0092$ and $p_5\simeq0.0007.$
Hence the following questions arise:

\caixa{
\begin{problem} Let $p_n$ be the probability that the zero solution
of the $n$-th order linear random differential equation
\eqref{eqdiflineal} is globally asymptotically stable. Find the
asymptotic expansion of $p_n$ at $n=\infty.$ Is it  true that the
sequence $p_n$ is strictly decreasing?
\end{problem}}

Similar problems can be consider for the random difference equations
of order $n$ of type
\begin{equation}\label{eq:eed}
A_nx_{k+n}+A_{n-1}x_{k+n-1}+\cdots +A_1 x_{k+1}+A_0x_k=0,
\end{equation}
where all the coefficients are again i.i.d.\!  random variables with
$N(0,1)$ distribution. In this situation, the global asymptotic
stability  happens when all the zeros  of the associated random
characteristic polynomial $Q(\lambda)$ have modulus smaller than 1.
Recall that  this property is characterized by the so-called  Jury
criterion. In fact, it is possible to get Jury conditions from
Routh--Hurwitz conditions and viceversa by using the M\"{o}bius
transformation that sends the left hand part of the complex plane
into the complex ball of radius 1 and its inverse.

If we call  $q_n$ the probability of the zero solution of
\eqref{eq:eed} to be globally asymptotically  stable, for instance,
$q_1=1/2,$  $q_2=\arctan(\sqrt{2})/\pi\simeq0.304,$
$q_3\simeq0.172,$ $q_4\simeq0.103,$ and $q_5\simeq0.059,$ see again
\cite{CimGasMan2021}.

\section{Some geometrical problems}\label{se:gp}

In this section we present three problems with a geometric flavour.

\subsection{Triangular billiards}

Consider  a mathematical ideal convex billiard with a smooth
$\mathcal{C}^1$ boundary. A punctual ball  moves on it alternating
between
 free motion (following a straight line) and
specular reflections from its boundary. When the particle hits the
boundary it reflects from it without loss of speed with an elastic
collision. Then, it is know that  there are always periodic
trajectories (\cite{Kat1995}). The number of times that a periodic
trajectory touches the boundary before closing is called its period.

If we consider a convex billiard, but with a polygonal boundary it
is natural to wonder if periodic trajectories also always exist. For
this type of billiards if the punctual ball arrives to a corner the
trajectory stops and, of course, it is not periodic. In fact, even
for triangular billiards this question is an open problem.

\caixa{
\begin{problem} Do all triangular billiards have some periodic
trajectory?\end{problem}}

\begin{figure}[h]
    \centering\includegraphics[scale=0.45]{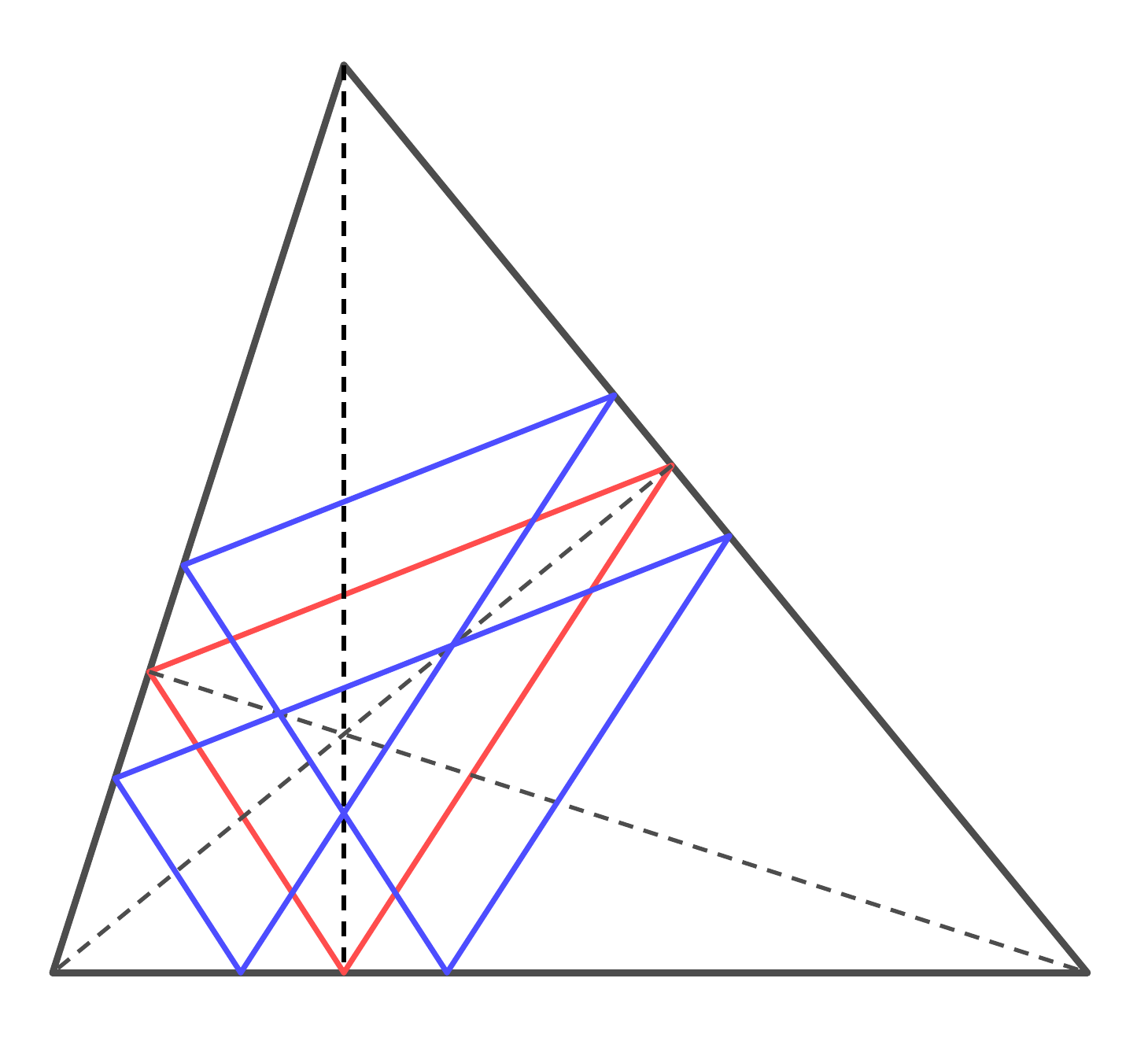}
\caption{Trajectories with periods 3 and 6 for an acute triangular
    billiard.}\label{fig:triangle}
\end{figure}

For many triangular billiards the answer is yes, see
\cite{Art2015,HalHun2000}. For instance, this is the case when the
boundary of the billiard is an acute triangle. In this case there is
always a periodic trajectory of period 3. It is formed by the
triangle that has as vertices the basis points at the boundaries of
the three heights, see Figure \ref{fig:triangle}. Sometimes this
trajectory is called Fagnano's trajectory, because he found it in
1775 for solving another problem: {\it find the inscribed triangle to
an acute triangle with smaller length}, see Figure~\ref{fig:triangle}. The answer is also
affirmative for rectangular triangles, isosceles triangles
(\cite{Cip1995}), for obtuse triangles with no angle larger that 100
degrees (\cite{Sch2009}), and also for rational triangles
(\cite{Bos1998}). Recall that a triangle is called rational if all
its angles are rational multiples of $\pi.$

\subsection{An extended Poncelet's problem.}

Given  two convex algebraic ovals, $\gamma$ and $\Gamma,$  as in
Figure \ref{fi:oval}, consider the Poncelet's map $P$  from the
exterior curve $\Gamma$ into itself, also introduced in this figure.
The iteration of this procedure is sometimes called Poncelet's
procedure.

\begin{figure}[h]
\centerline{\includegraphics[scale=0.8]{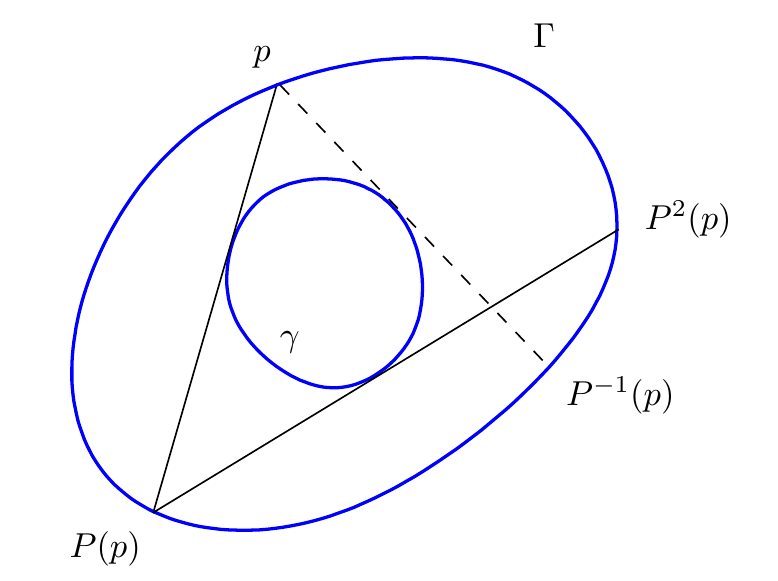}}
\caption{Poncelet's map.}\label{fi:oval}
\end{figure}

The name for this map is introduced in \cite{CimGas2010} because
Poncelet, a French engineer and a mathematician, considered it for
first time when both ovals are ellipses, while he was prisoner in
Saratov (Russia) during 1812-1814. He proved the following nice
theorem that is illustrated in Figure~\ref{fi:oval2} when $n=3.$

\begin{theo}[Poncelet's theorem] If given an initial point $p$ on the exterior ellipse $\Gamma$ the  Poncelet's
procedure  closes for first time after  $n$ steps, then the  same
happens for any other initial condition.
\end{theo}

\begin{figure}[h]
\centerline{\includegraphics[scale=0.9]{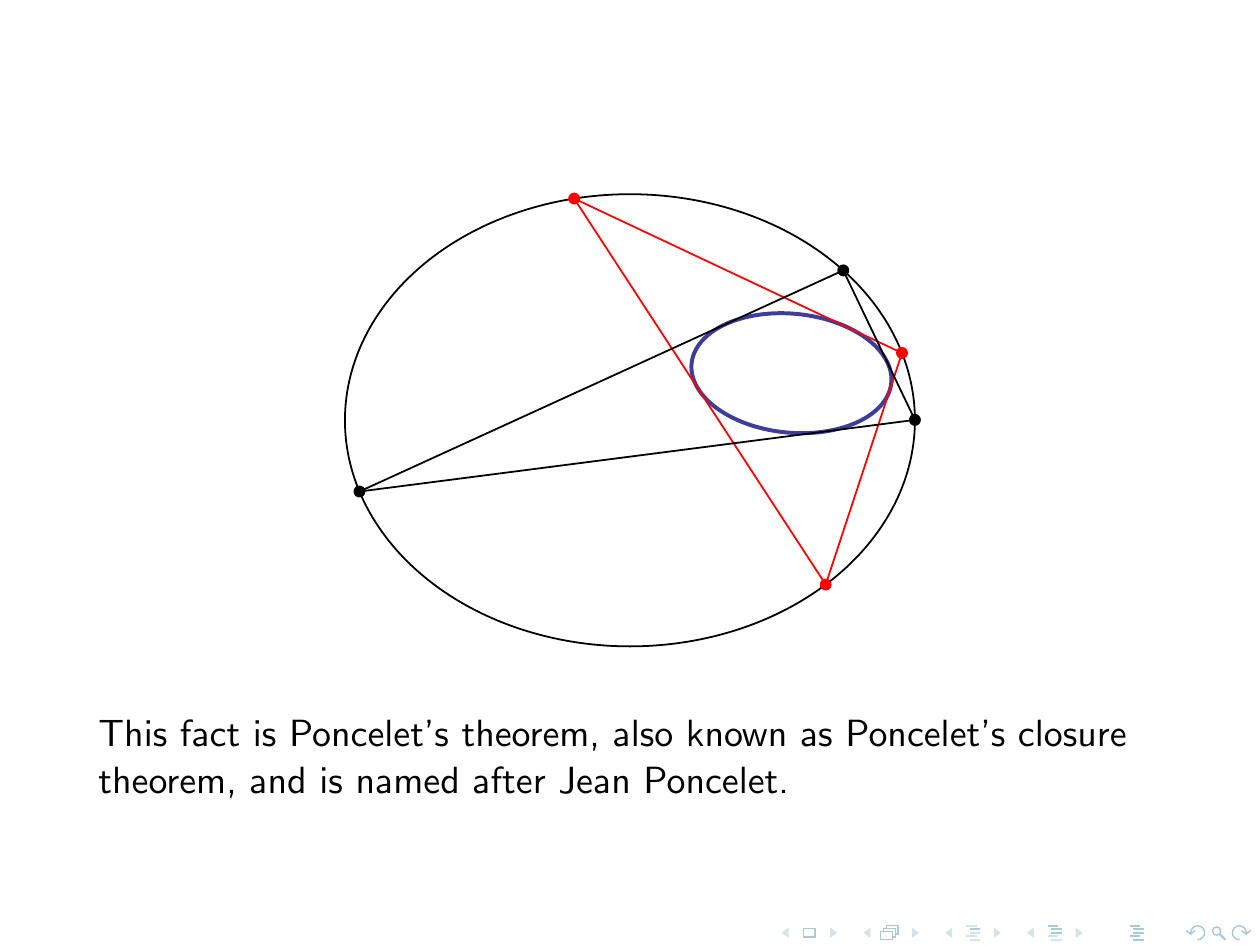}} \caption{Two 3
steps closed Poncelet's trajectories.}\label{fi:oval2}
\end{figure}

There are several  proofs of the above theorem, see \cite{Tab1995}.
In fact, in the language of dynamical systems the above result can
be extended giving the following theorem:

\begin{theo} Let $P$ be the Poncelet's map between two ellipses. Then $P$ is
 conjugated  with a  rotation $R$  of the
circle. In particular,
\begin{itemize}

\item[(i)] if the  rotation number of $R$ is rational  then
all points  are  periodic for $P$ and with the same period
(original Poncelet's theorem).

\item[(ii)] if the  rotation number of $R$ is irrational  then
all points of any orbit of $P$ are  dense  on the exterior ellipse.

\end{itemize}

\end{theo}

In \cite{CimGas2010} it is  proved the following result, that shows
that Poncelet's result is not true for all ovals.

\begin{prop} Consider $\gamma=\{x^{2n}+y^{2n}=1\}$ and
$\Gamma=\{x^{2m}+y^{2m}=2\}$ with $n,m\in\N.$ Then their associated
Poncelet's map has rotation number $1/4$ and it is conjugated to a
rotation if and only if $n = m = 1.$
\end{prop}

A natural question is the following:

\caixa{
\begin{problem}
Are there two irreducible algebraic curves of degrees $n$ and $m,$
with  $n+m>4$, having each one of them an  oval, for which the
Poncelet's map $P$ is well defined and it  is conjugated  to a
rotation of the circle?
\end{problem}}

This problem is somehow reminiscent of the  classical Birkhoff's
conjecture. Recall that it claims that the boundary of a strictly
convex integrable billiard table is necessarily an ellipse (or a
circle as a special case). Recently, in \cite{KalSor2018} it is
proved a local version of this conjecture: a small integrable
perturbation of an ellipse must be an ellipse.

Inspired on the above point of view, we propose next local version
of the above question:

\caixa{
\begin{problem} Consider the ovals $\gamma=\{x^2+y^2-1=0\}$ and
$\Gamma_\varepsilon=\{p_2(x,y)+\varepsilon p_m(x,y)=0\}$ where
$\Gamma_0$ is an ellipse that surrounds $\gamma,$
$p_2(x,y)+\varepsilon p_m(x,y)=0$ is an irreducible curve, with
$p_m$ a polynomial of degree $m\ge3,$  and $\varepsilon$ is a small
parameter. Is it true that the Poncelet's map associated to both
ovals is conjugated to a rotation if and only if $\varepsilon=0$?
\end{problem}}

\subsection{Loewner's conjecture} By using the complex notation
introduced in Section \ref{se:equi} we  consider the Cauchy--Riemman
operator
\[
\frac{\partial }{\partial \bar z}=\frac12\left(\frac{\partial
}{\partial x}+{\rm i}\frac{\partial }{\partial y}\right).
\]
Then
\[
\frac{\partial^2 }{\partial \bar z^2}=\frac14\left(\frac{\partial^2
}{\partial x^2}-\frac{\partial^2 }{\partial x^2}+2{\rm
i}\frac{\partial^2 }{\partial x\partial y}\right).
\]
Similarly, given $1<n\in\N,$ we can define $\frac{\partial^n
}{\partial \bar z^n}.$ Given a neighbourhood of the origin
$\mathcal{U}\subset\R^2$ and a class $\mathcal{C}^{n+1}$ function
$f:\mathcal{U}\to\R$ such that $f(0,0)=0,$ we look at the planar
differential equation
\begin{align}\label{eq:loe}
\dot x= 2^n\operatorname{Re}\left( \frac{\partial^n }{\partial \bar
z^n} f(x,y)\right),\quad \dot y=2^n\operatorname{Im}\left(
\frac{\partial^n }{\partial \bar z^n}f(x,y) \right).
\end{align}
For instance for $n=1$ and $2$ we have
\begin{align*}
&(\dot x,\dot y)= (f_x(x,y),f_y(x,y)),\qquad n=1,\\ &(\dot x,\dot
y)= (f_{xx}(x,y)-f_{yy}(x,y),2f_{x,y}(x,y)),\qquad n=2.
\end{align*}

Recall that the index is an integer number associated to any
isolated equilibrium point of a planar differential equation that
measures the number of turns of its associated vector field near it,
see \cite{DumLliArt2006} for a precise definition. When this
isolated equilibrium point admits a finite sectorial decomposition
(this always happens for instance in the analytic case, for
non-monodromic singularities, see \cite{IlyYak2008}) and $e, h,$ and
$p$ denote its number of elliptic, hyperbolic, and parabolic
sectors, respectively, then  the index is $1+ \frac{e -h} 2,$ due to
Poincar\'{e}'s index formula.

According to \cite{Tit1973}, next conjecture was proposed by Loewner
around 1950, see also \cite{LliMar2012}.

\caixa{
\begin{problem}[Loewner's conjecture] Assume that the origin is an
isolated equilibrium point of the differential equation
\eqref{eq:loe} and that $f$ is analytic at this point. Then the
index of the associated vector field at the origin is is not greater
than  $n.$
\end{problem}}

Another related conjecture is Carath\'{e}odory's conjecture, that
asserts that every smooth convex embedding of a 2-sphere in $\R^3,$
i.e. an ovaloid, must have at least two umbilics. Recall that for
any surface in $\R^3,$ the eigenspaces of the second fundamental
form define two orthogonal line fields (principal directions) whose
singularities are exactly the so-called umbilics. It is known that
if Loewner's conjecture is true when $n=2$  the Carath\'{e}odory's
conjecture is also  true in the analytic case, see for instance
\cite{And2003,GutSot1998}.

Loewner's conjecture is true for $n=1$ and several authors have
proved it for $n=2,$ although there are several wrong proofs, see
some comments in \cite{GutSan1997,GutSot1998}. Anyway, it would be
very interesting to have simple proofs of Loewner's conjecture for
$n=2$ and to further investigate it, also for functions of class
$\mathcal{C}^{n+1}.$

\section{Problems involving polynomials}\label{se:pol}

\subsection{A moments problem.}

Arno van den Essen, the author of the interesting monograph
\cite{Ess2000} about the Jacobian conjecture, which recall that
asserts the bijectivity of the complex polynomial maps with constant
Jacobian, call our attention to the following question:

\caixa{
\begin{problem} Let $f(x_1,x_2,\ldots,x_n)$ be a polynomial in $\mathbb{C}[x_1,x_2,\cdots,x_n]$ such that the following
 moment conditions hold
\[
M_m:=\int _0^1\int _0^1 \cdots \int _0^1  f^m(x_1,x_2,\ldots,x_n)\,
{\rm d}x_1\,{\rm d}x_2\,\cdots {\rm d}x_n=0,\quad m\ge 1.
\]
Is it true that $f=0$?
\end{problem}}

This question and some extensions where some especial weights are
added to the above integrals, is very related with the Jacobian
conjecture, see \cite{DerEssWen2017, Fra2014}. Obviously it can also
be extended to wider classes of maps $f.$

In fact, the answer when $n=1$ is yes, see \cite{Pak2013}, but
anyway, in this case it would be nice to get a simple direct proof.
Notice  that when $n=1$ for $|t|$ small enough
\[
\frac1{1-tf(x)}=1+\sum_{m=1}^\infty (tf(x))^m,
\]
and since the convergence is uniform, we get that under the above
hypotheses,
\[
\int_0^1\frac1{1-tf(x)}\,{\rm d}x=1+\sum_{m=1}^\infty t^m\int_0^1
f^m(x)\,{\rm d}x= 1+\sum_{m=1}^\infty M_mt^m=1.
\]

Hence the answer of the above question when $n=1$ is equivalent to
prove that if $f(x)$ be a polynomial in $\mathbb{C}[x]$ such that
for all $|t|$ small enough:
\[
\int_0^1\frac1{1-tf(x)}\,{\rm d}x=1,
\]
then $f=0.$

Also an interesting related question is the following:

\caixa{
\begin{problem} (i) Let $f(x)$ be a polynomial in $\mathbb{C}[x]$ with  $k$ monomials. Is there a value $N(k)$ such that if the following
 finite set of moment conditions hold:
\[
M_n:=\int _0^1 f^n(x)\, {\rm d}x=0,\quad  1\le n\le N(k),
\]
then $f=0$?

(ii) If the answer is yes, find $N(k)$ or a good upper bound of this
number.
\end{problem}}

This type of questions  also appear in some classical problems for
Abel differential equations, where the vanishing of certain moments
imply the solution of the center problem, see for instance
\cite{CimGasMan2013}.

\subsection{Around Kouchnirenko's conjecture.}\label{se:few}

Descartes' rule implies that a 1-variable real polynomial with $m$
monomials has at most $m-1$ simple  positive real roots.

The  Kouchnirenko's conjecture   was posed as an
attempt to  extend  this rule to the
several variables  context. In the 2-variables case this conjecture
said that:

 \begin{quote}A real polynomial system
$f_1(x,y)=f_2(x,y)=0$ would have at most $(m_1-1)(m_2-1)$ simple
solutions with positive coordinates, where $m_i$ is the number of
monomials of each polynomial $f_i,$ $i=1,2.$
\end{quote}

 This conjecture was stated by
A.~Kouchnirenko in the late 70's, and published in  \cite{Kho1980}
in 1980. In 2002, in \cite{Haa2002}  a family of counterexamples
given by two trimonomials, being their minimal degree  106, was
constructed. In 2003 a much simpler family of counterexamples was
presented in \cite{LiRojWan2003} again formed by two trimonomials,
but of degree $6$. Both have exactly  $5$ simple solutions with
positive coordinates instead of the  $4$ predicted by the
conjecture. A similar counterexample  is:

\begin{prop}[\cite{Gas2020}]\label{pro:kou}  The bivariate trinomial system
\begin{equation}\label{eq:kou}
\left\{\begin{array}{l}
P(x,y):=x^{6}+\dfrac{61}{43} y^3-y=0,\\[12pt]
Q(x,y):=y^{6}+\dfrac{61}{43} x^3-x=0,
\end{array}\right.
\end{equation}
has 5 real simple solutions with positive entries.
\end{prop}

It is not difficult to find numerically  5 approximated  solutions
of the system. They are $(\widetilde x_1,\widetilde x_5)$,
$(\widetilde x_2,\widetilde x_4)$, $(\widetilde x_3,\widetilde
x_3)$, $(\widetilde x_4,\widetilde x_2)$, $(\widetilde
x_5,\widetilde x_1)$, where $\widetilde x_1=0.59679166,$ $\widetilde
x_2= 0.68913517,$  $\widetilde x_3= 0.74035310,$ $\widetilde x_4=
0.77980435$ and  $\widetilde x_5= 0.81602099.$ A proof that these
solutions actually exist follows by using Poincar\'{e}-Miranda theorem,
see \cite{Gas2020}. Recall that this theorem is an  extension  of
the classical Intermediate Value theorem (or Bolzano's theorem) to
higher dimensions. It was stated by H.~Poincar\'e in 1883 and 1884,
and proved by himself in 1886. In 1940, C.~Miranda (\cite{Mir1940})
re-obtained the result as an equivalent formulation of Brouwer fixed
point theorem:

\begin{theo}[Poincar\'{e}-Miranda (\cite{Maw2019})]
       Set $\mathcal{B}=\{\mathbf{x}=(x_1,\ldots,x_n)\in\R^n\,:\,L_i<x_i<U_i, 1\leq i\leq n\}$. Suppose that
         $f=(f_1,f_2,\ldots,f_n):\overline{\mathcal{B}}\rightarrow \R^n$ is
         continuous,  $f(\mathbf{x})\neq\mathbf{0}$
           for all $\mathbf{x}\in\partial \mathcal{B}$, and for  $1\leq i\leq
       n,$
       \begin{align*}
       &f_i(x_1,\ldots,x_{i-1},L_i,x_{i+1},\ldots,x_n)\leq 0\, \mbox{ and }\\
        &f_i(x_1,\ldots,x_{i-1},U_i,x_{i+1},\ldots,x_n)\geq 0.
        \end{align*}
       Then, there exists $\mathbf{s}\in\mathcal{B}$ such that  $f(\mathbf{s})= \mathbf{0}$.
       \end{theo}
       In Figure \ref{fig:pm} we illustrate the hypotheses of the
       theorem for $n=2.$

\begin{figure}[h]
\centering\includegraphics{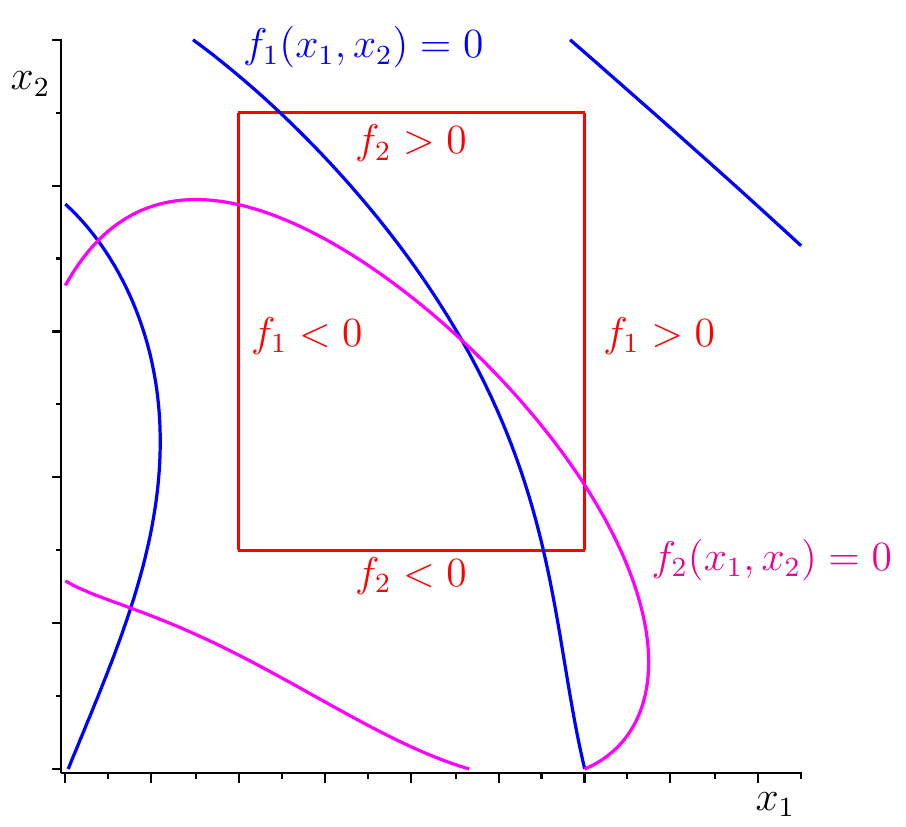}\caption{A Poincar\'{e}-Miranda
box.}\label{fig:pm}
\end{figure}

To prove Proposition \ref{pro:kou} we consider the following 5
intervals, with $\widetilde x_i\in I_i,$
\begin{align*}
I_1&=\left[{\frac{1}{2}},{\frac{1619}{2500}}\right],\,
I_2=\left[{\frac{1619}{2500}},{\frac{18}{25}}\right],\,
I_3=\left[{\frac{18}{25}},{\frac{75857}{100000}}\right],\\
I_4&=\left[{\frac{75857}{100000}},{\frac{4}{5}}\right],\,
I_5=\left[{\frac{4}{5}},{\frac{83}{100}}\right].
\end{align*}
and prove that our system  has  5 actual  solutions $( x_1, x_5)$,
$( x_2, x_4)$, $( x_3, x_3)$, $( x_4, x_2)$, $( x_5, x_1)$, with
$x_i\in I_i.$ By Descartes' rule we know that there is exactly  one
simple positive real root of $P(x,x).$ The corresponding $(x_3,x_3)$
is in in $I_3\times I_3$. By the symmetry of the system, if
$(x^*,y^*)$ is one of its solutions then  $(y^*,x^*)$  also is.
Finally, we can prove the existence of  two  more solutions (and so,
their symmetric ones) by using the Poincar\'{e}-Miranda theorem in the
boxes $I_1\times I_5$ and $I_2\times I_4$, see again \cite{Gas2020}.

In  \cite{Dic2007,LiRojWan2003}   the authors prove that any
bivariate trinomial system $ m_1=m_2=3$  has  at most 5  real simple
solutions with positive entries and henceforth this bound is sharp.
 A very interesting problem is:

\caixa{
\begin{problem}
Find a  reasonable (or sharp) upper bound  in terms of $m_i$ for the
maximum number of simple solutions  with positive coordinates, for a
real polynomial system $f_1(x,y)=f_2(x,y)=0,$ where  $m_i, i=1,2,$
is the number of monomials of each  $f_i$.
\end{problem}}

Not sharp upper bounds are known from the nice approach of
Khovanski\u{\i} who was a pioneer in 1980 in the study of the so
called fewnomials (\cite{Kho1980}). His general upper bound is as
follows: given a system of $n$ real polynomial equations in $n$
variables with a total of $n+k+1$ distinct monomials possesses at
most $2^{n+k\choose2}(n+1)^{n+k}$ nondegenerate solutions with
positive entries. This upper bound has been improved in
\cite{BihRojSot2008} decreasing it until $\frac{{\rm e}^2+3}4
2^{k\choose 2}n^k.$ In fact, in \cite{BihRojSot2008} when $n>k,$ an
example with $[\frac{n+k}k]^k$ nondegenerate solutions with positive
entries is given, showing that for $k$ fixed and $n$ big enough this
last upper bound is almost asymptotically sharp.

To illustrate that the above results are not sharp enough, let us
apply them to the planar trinomial situation ($n=2$) where  the
sharp upper bound is 5. In principle, it seems that $k=3$, because
there are 6 involved monomials, but notice that the number of
solutions in the first quadrant remains unchanged when we multiply
any of the equations by any monomial $x^iy^j.$ Hence, before
applying the given bounds, we can do this modification in a
convenient way to force to coincide two of them. In this way the
value $k$ can be assumed to be $k=2.$ Hence, Khovanski\u{\i}'s upper
bound gives $2^63^4=5184$ and its improvement gives $20.$

Recall again that by using  Descartes' rule it is easy to answer
this last problem  in one variable. Moreover, the $m-1$
corresponding to  positive  solutions implies a global upper bound
of $2m-1$ solutions: $m-1$ positive roots, $m-1$ negative ones and,
eventually, the root $0,$ that can be  multiple, with any
multiplicity.

It is natural to wonder if the following  modified
Kouchnirenko's bound works.

\caixa{
\begin{problem}
Is  $(2m_1-1)(2m_2-1)$  the maximum number of simple
solutions of a real polynomial system $f_1(x,y)=f_2(x,y)=0,$  where
$m_i$ is the number of monomials of each $f_i$?
\end{problem}
}

It is very easy to find examples of uncoupled systems having
 $(2m_1-1)(2m_2-1)$ simple solutions. For instance, for $m_1=m_2=3,$ then $ (2m_1-1)(2m_2-1)=25$. Consider
$(x^2-1)(x^2-4)x=x^5-5x^3-x.$ Then the system
\begin{equation*}
\left\{\begin{array}{l}
x^5-5x^3-x=0,\\
y^5-5y^3-y=0,
\end{array}\right.
\end{equation*}
has the  25  simple solutions $(x_i,x_j)$ with $x_i,
x_j\in\{-2,-1,0,1,2\}.$ Similarly, the system
\begin{equation*}
\left\{\begin{array}{l}
x^{5+r}-5x^{3+r}-x^{1+r}=0,\\
y^{5+s}-5y^{3+s}-y^{1+s}=0,\quad s>0,r>0,
\end{array}\right.
\end{equation*}
has  16  simple solutions and  9  multiple
ones.

Another example with 25  solutions can be constructed from our
counterexample \eqref{eq:kou}. It is obtained by taking the
equations $yP(x^2,y^2)=0 $ and $xQ(x^2,y^2)=0,$ giving
\begin{equation*}
\left\{\begin{array}{l} \left(x^{12}+\dfrac{61}{43}
y^6-y^2\right)y=x^{12}y+\dfrac{61}{43}
y^7-y^3=0,\\[12 pt]
\left(y^{12}+\dfrac{61}{43} x^6-x^2\right)x=y^{12}x+\dfrac{61}{43}
x^7-x^3=0,
\end{array}\right.
\end{equation*}
which has  $4\times 5=20$ solutions, 5 in each quadrant, plus 5 more
on the axes: $(0,0),$ $(\pm x^*,0)$ and $(0,\pm y^*),$ for some
$x^*$ and $y^*.$ Again  $25$ solutions and here $(0,0)$  is not  a
simple solution.

Of course there are also natural extensions to $n$ equations and $n$
variables of the above problems.

\subsection{Casas-Alvero's conjecture.}

Casas-Alvero arrived to the next conjecture at the turn of this
century, when he was working  trying to obtain an irreducibility
criterion for two variable power series with complex coefficients
(\cite{Cas2001}).

\caixa{
\begin{problem} [Casas-Alvero's conjecture]
If a complex  polynomial
$P$ of degree $n>1$ shares roots with all its derivatives, $P^{(k)},
k=1,2\ldots, n-1,$ then there exist two complex numbers, $a$ and
$b\ne0,$  such that  $P(z)=b(z-a)^n.$
\end{problem}}

 Notice that, in principle, the
common root between $P$ and each $P^{(k)}$ might depend on~$k.$ Several
authors have got partial answers, but as far as we know,  the conjecture remains open.
 For $n\le4$ the
conjecture is a simple consequence of the wonderful Gauss--Lucas
theorem, that asserts that the complex roots of $P'(z)$ are in the
convex hull of the roots of $P(z).$ It is also known that the
conjecture is true  for low degrees and also when~$n$ is
$p^m,2p^m,3p^m,$ or $4p^m,$ for some prime number $p$ and $m\in\N.$
 Nowadays, the first cases left open are $n=24,28,$ or $30.$ See the nice survey \cite{DraJon2011} and its references.

It is also known that if the conjecture holds in $\C,$ then it is
true over all fields of characteristic~$0.$ On the other hand, it is
not true over all fields of characteristic $p,$ see~\cite{BLSW2007}.
For instance, consider $P(x)=x^2(x^2+1)$ in characteristic $5$ with
roots $0,0,2$ and $3.$  Then $P'(x)=2x(2x^2+1),$ $P''(x)=12x^2+2=
2(x^2+1),$ and  $P'''(x)=4x$ and all them share roots with $P.$

Adding the hypotheses that $P$ is a real polynomial and all its
roots are real, the conjecture has a real counterpart, that also
remains open. It says that $P(x)=b(x-a)^n$ for some real numbers $a$
and $b\ne0.$ For this case,  from Rolle's theorem it
 follows easily  for $n\le4.$ However, it is not difficult to see that this tool does not suffice to prove it for bigger $n.$

Remarkably,  in the real case it is proved in \cite{DraJon2011} that
if  the condition for one of the derivatives of $P$  is removed,
then there exist polynomials, different from $b(x-a)^n,$ satisfying
the remaining $n-2$ conditions. The construction presented in that
paper of some of these polynomials  is a nice consequence of the
Brouwer's fixed point theorem in a suitable context.

Recently in \cite{CimGasMan2020} it is shown that the natural
extension of this real conjecture to the smooth world is not true.
There it is considered the following problem:  Fix $1<n\in\N$ and
let $F$ be a class $\mathcal{C}^n$  real function such that
$F^{(n)}(x)\ne0$ for all $x\in\R,$ having $n$ real zeroes, taking
into account their multiplicities. Assume that $F$ shares zeroes
with all its derivatives, $F^{(k)}, k=1,2\ldots, n-1.$ Is it true
that $F(x)=b (f(x))^n$ for some $0\ne b\in\R$ and some $f,$ a class
$\mathcal{C}^n$ real function, that has exactly one simple zero?

The answer for the above problem is ``yes'' for $n\le 4$ and ``no''
for $n=5.$  More concretely, it is proved that there exists $r>1$
such that if we consider
\[
F(x)=\int_{0}^x \int_0^u\int_1^w\int_c^z\int_1^y \big(r-\sin(t)\big)
\,{\rm d} t \,{\rm d} y \,{\rm d} z \,{\rm d} w \,{\rm d} u
\]
it holds that $F$ has the five zeroes $0,0,1,c,d$ satisfying $0<1<c<d,$
\begin{equation*}\label{cond}
F'(0)=0,\, F''(1)=0,\, F'''(c)=0,\, F^{(4)}(1)=0,\,\mbox{ and }\,
F^{(5)}(x)=r-\sin(x)>0,
\end{equation*}
and  $F$ is not of the proposed form.

\section{Some conjectures with a dynamical flavour}\label{se:con}

One of the most famous open problems is the  so-called $3x+1$
conjecture or Collatz problem (\cite{Lag2010}). Recall that it
assures that for any
  $x_0\in\N,$ the sequence defined by
\begin{equation*}
x_{n+1}=g(x_n)=\begin{cases} 3x_n+1,\quad &\mbox{when}\quad x_n
\quad \mbox{is odd},\\[0.1cm] {x_n}/{2},\quad &\mbox{when}\quad x_n \quad
\mbox{is even},
\end{cases}
\end{equation*}
arrives after finitely many steps to the 3-periodic behaviour
$4,\,2,\,1,\,4,\,2,\,1,\ldots$

We end this paper with three similar but less known conjectures. The
first one was proposed by N. Sloane (\cite{Slo1973}) in a journal of
recreational mathematics.

\caixa{
\begin{problem}[Conjecture of multiplicative persistence] Given $n\in\N,$ let $\Pi(n)\in\N$ be the  product of
all its digits. Set $\operatorname{Pm}(n)\in\N$ for the first
positive integer such that $\Pi^{\operatorname{Pm}(n)}(n)=
\Pi^{\operatorname{Pm}(n)+1}(n),$ where $\Pi^0=\operatorname{Id}$ i
$\Pi^k(n)=\Pi(\Pi^{k-1}(n)).$ Is it true that for all $n\in\N,$
$\operatorname{Pm}(n)\le 11$?
\end{problem}}

For instance, for  $n=68889,$
$\Pi(68889)=6\times8\times8\times8\times9=27648$ and
\[
68889\to 27648 \rightarrow 2688 \rightarrow768
\rightarrow336\rightarrow54\rightarrow20\rightarrow0\rightarrow0\rightarrow0\rightarrow\cdots.
\]
Then  $\operatorname{Pm}(6889)=7$, because
$\Pi^7(6889)=\Pi^8(6889)=0$  is the first coincidence. The smallest
numbers with respective multiplicative  persistence $1, 2,\ldots,
11$ are
\begin{align*}
10, 25, 39, 77, 679, 6788, 68889,
2677889, 26\,888\,999, 3\,778\,888\,999, 277\,777\,788\,888\,899.
\end{align*}
There are not examples with higher multiplicative  persistence for
$n<10^{233}.$

Notice that for instance $\Pi(M) = 2^{19}3^47^6.$ In general, a
simple first observation already pointed out in \cite{Slo1973} is
that the prime factors of any $\Pi(n)$ with persistence bigger
than~$3$ must be either $2^i3^j7^k$ or  $3^i5^j7^k.$ Therefore, it
suffices to study the  persistence of these numbers. This is so,
because $\Pi(n)=2^i3^j5^k7^m,$ is a product of one digit prime
numbers,
 with all the exponents greater or equal than zero, and moreover if  $2$ and $5$ appear
  together  $2^i3^j5^k7^m$  ends with zero and then $\Pi^2(n)=\Pi^3(n)=0.$ This problem has been recently
   extended to other basis and studied from a dynamical systems point of view in \cite{FarTre2014}.

\caixa{
\begin{problem}[196 conjecture]  Let $f:\N\rightarrow\N$ be
defined as $f(n)=n+\operatorname{rev}(n),$ where
$\operatorname{rev}$ is the map that reverses the order of the
digits of $n.$ Then, there are infinitely many natural numbers $n$
such that  $f^k(n),$ for $0<k\in\N,$ where $f^0=\operatorname{Id}$
and $f^k(n)=f(f^{k-1}(n)),$ is never  a palindromic number.
Moreover, the smallest of these numbers is $196.$
\end{problem}}

For instance, if $n=183,$ $f(183)=183+381=564,$ and
\[
183\rightarrow564\rightarrow 564 + 465 = 1029\rightarrow 1029 + 9201
= 10230\rightarrow10230 + 3201 = 13431,
\]
that is a palindromic number. Starting with $n=89,$ we need 24
iterations to arrive to a palindromic number, that is
$88132\,000\,23188.$ Until today, starting with $n=196$  no
palindromic numbers have been found yet, see \cite{Nis2012}. It is
not clear the origin of this problem. The first reference goes back
to Lehmer in 1938 (\cite{Leh1938}). The question recovered some
interest after the paper \cite{Tri1967}, published in 1967.
Sometimes the numbers $n$ such that $f^k(n)$ is never a palindromic
number are called Lychrel's numbers (it is an acronym  of the name
Cheryl).

The first numbers that could be Lychrel's numbers are
\[
196, 295, 394, 493, 592, 689, 691, 788, 790, 879, 887, 978, 986,
\ldots
\]
In Figure \ref{fig:capicua} we plot the function $h$ that assigns to
each $n\in\N$ the minimum value  $h(n)\le1000$ such that $f^{h(n)}$
is a palindromic number, or 1000 when none of the first 1000
iterates is a palindromic number. I thank Antoni Guillamon for
sharing with me the Maple code that I have used to generate  this
figure. For the sake of clarity, we restrict the plot to the strip
$1\le h(n)\le 40.$ Its spikes correspond to the 13 values of the
above list. Notice also that for all the other values of $n,$ the
function $h(n)$ is at most $24.$

\begin{figure}[h]
    \centering\includegraphics[scale=0.7]{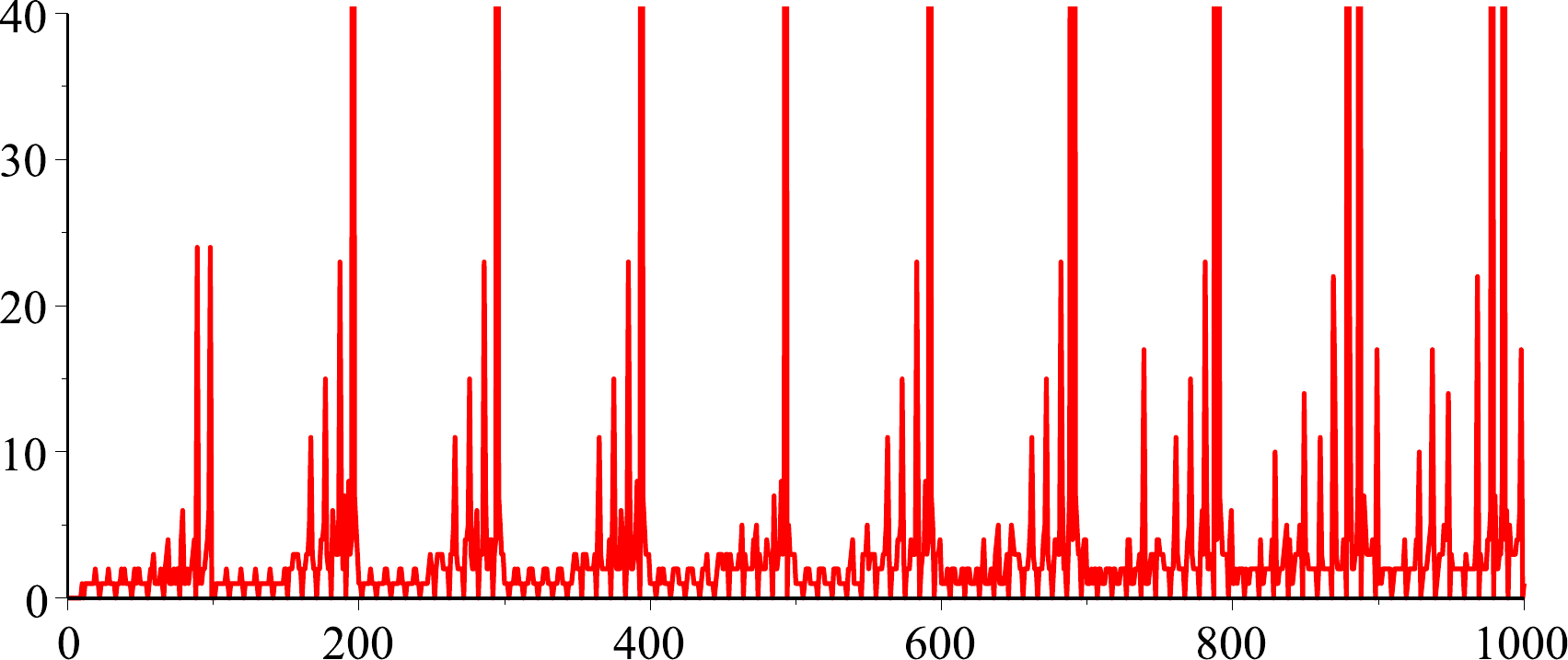}
\caption{Possible  Lychrel's numbers.}\label{fig:capicua}
\end{figure}

Although in basis 10 the existence of Lychrel's numbers is an open
problem, it is not difficult to find some of them in other basis.
For instance in basis 2, if we take
 $n= 10110_2$ is one of them. This is so, because $f^4(n)=
 10110100_2$,
 \[
10110_2\rightarrow 10110_2+01101_2=
100011_2\rightarrow1010100_2\rightarrow1101001_2\rightarrow
 10110100_2,
 \]
 $f^8(n)= 1011101000_2,$
 $f^{12}(n)=  101111010000_2,$ and in general (\cite{Bro1968}) after $4m$ iterates we arrive to a number that starts with 10,
 after has $m+1$ ones, then  $01$ and it ends with $m+1$ zeroes.

\caixa{
\begin{problem}[Singmaster's conjecture] There is a value $S\in\N$ such that any number different from $1$ appears in
the Pascal's triangle at most $S$ times.
\end{problem}}

This conjecture was proposed by  David Singmaster in 1971, see
\cite{Sin1971}. He already proved in 1975 that there are infinitely
many values that appear 6 times. One of them is 120,
 \[
120={120\choose1}={120\choose119}={16\choose2}={16\choose14}={10\choose3}={10\choose7}.
\]
 In fact, it holds that
 \[
m={m\choose 1}={m\choose m-1}={n+1\choose k+1}={n+1\choose
n-k}={n\choose k+2}={n\choose n-k-2},
 \]
 where for any $i\in\N,$ $n=F_{2i+2}F_{2i+3}-1,$
 $k=F_{2i}F_{2i+3}-1,$ $m={n+1\choose k+1},$ and $F_j$ is the $j$-th Fibonacci number,
 being $F_0=0$ and $F_1=1,$ see \cite{Sin1975}.
The only known number that appears  8 times is
 \[
3003={3003\choose1}={3003\choose3002}={78\choose2}={78\choose76}
={15\choose5}={15\choose10}={14\choose6}={14\choose8}.
 \]
Hence, if the conjecture holds, $S\ge8.$ It seems that
 Singmaster thought that $S$ could be $10$ or $12,$ although many people
 starts thinking that $S=8.$ Notice that any
$1<m\in\N$ appears  finitely many times because this value can only
appear in the first  $m+1$ files.

\subsection*{Acknowledgements} The author thanks Jos\'{e} Luis Bravo and Joan Torregrosa
for their feedback on previous versions of this paper.

This work has received funding from the Ministerio de Ciencia e
Innovaci\'{o}n (PID2019-104658GB-I00 grant) and the Ag\`{e}ncia de Gesti\'{o}
d'Ajuts Universitaris i de Recerca (2017 SGR 1617 grant).

%\bibliographystyle{acm}
%\bibliography{bibfile}

%%\bibliographystyle{plain}

%%\begin{thebibliography}{99}

%%\end{thebibliography}

\end{document}